\documentclass[10pt,reqno,final]{amsart}%
\usepackage[a4paper,left=2.8cm,right=2.8cm,top=2.8cm,bottom=3.5cm]{geometry}%
\usepackage[hidelinks,unicode=true]{hyperref}%
\usepackage[utf8]{inputenc}%
\usepackage[shortlabels]{enumitem}
\usepackage{mathtools,graphicx,todonotes,lmodern,enumitem,marginnote,amssymb}%
\usepackage[american]{babel}%
\graphicspath{{figs/}}%
\allowdisplaybreaks%
\numberwithin{equation}{section}%
\title[Threshold condensation to singular support]{Threshold condensation to singular support\\for a Riesz equilibrium problem}
\author{Djalil Chafaï}%
\address[DC]{DMA, École normale supérieure -- PSL, 45 rue d'Ulm, %
  F-75230 Cedex 5 Paris, France}%
\email{\url{mailto:djalil(at)chafai.net}}%
\urladdr{\url{http://djalil.chafai.net/}}%
\author{Edward B. Saff}%
\address[ES]{Center for Constructive Approximation, %
  1326 Stevenson Center, Vanderbilt University, Nashville, TN 37240, USA}%
\email{\url{mailto:Ed.Saff@Vanderbilt.Edu}}%
\urladdr{\url{https://my.vanderbilt.edu/edsaff/}}%
\author{Robert S.\ Womersley}%
\address[RW]{School of Mathematics and Statistics, %
  University of New South Wales, %
  Sydney NSW 2052, Australia}%
\email{\url{mailto:R.Womersley@unsw.edu.au}}%
\urladdr{\url{https://web.maths.unsw.edu.au/~rsw/}}%
\date{Spring 2022, revised Autumn 2022, Winter 2023, compiled \today}
\newtheorem{theorem}{Theorem}[section]%
\newtheorem{lemma}[theorem]{Lemma}%
\newtheorem{proposition}[theorem]{Proposition}%
\theoremstyle{remark}
\newtheorem{remark}[theorem]{Remark}%

\newcommand{\DOT}[2]{#1\cdot#2}
\newcommand{\SBRA}[1]{\left[#1\right]}
\newcommand{\PAR}[1]{\left(#1\right)}
\newcommand{\ABS}[1]{\left|#1\right|}

\newcommand{\dd}{\mathrm{d}}

\newcommand{\dS}{\mathbb{S}}

\newcommand{\dR}{\mathbb{R}}
\newcommand{\Rd}{\dR^d}
\newcommand{\sph}{\dS^{d-1}}
\newcommand{\prob}{\mathcal{M}_1(\Rd)}
\newcommand{\ds}{\displaystyle}
\newcommand{\ind}{\mathbf{1}}

\newcommand{\meq}{\mu_{\mathrm{eq}}}
\newcommand{\hgp}[2]{{}_{#1}F_{#2}} 
\newcommand{\hg}{\hgp{2}{1}}
\newcommand{\xh}{\hat{x}}
\newcommand{\yh}{\hat{y}}
\makeatletter
\def\@MRExtract#1 #2!{#1}     
\renewcommand{\MR}[1]{
  \xdef\@MRSTRIP{\@MRExtract#1 !}%
  \href{http://www.ams.org/mathscinet-getitem?mr=\@MRSTRIP}{MR-\@MRSTRIP}}
\makeatother
\begin{document}
\begin{abstract}
  We compute the equilibrium measure in dimension $d=s+4$ associated to a
  Riesz $s$-kernel interaction with an external field given by a power of the
  Euclidean norm. Our study reveals that the equilibrium measure can be a
  mixture of a continuous part and a singular part. Depending on the value of
  the power, a threshold phenomenon occurs and consists of a dimension
  reduction or condensation on the singular part. In particular, in the
  logarithmic case $s=0$ ($d=4$), there is condensation on a sphere of special radius 
  when the power of the external field becomes quadratic. This contrasts with the case $d=s+3$ studied
  previously, which showed that the equilibrium measure is fully dimensional
  and supported on a ball. Our approach makes use, among other tools, of the
  Frostman or Euler\,--\,Lagrange variational characterization, the
  Funk\,--\,Hecke formula, the Gegenbauer orthogonal polynomials, and
  hypergeometric special functions.
\end{abstract}
\maketitle
%


\section{Introduction and main results}

In the present work, we determine the equilibrium measure in $\Rd$ associated
with a Riesz $s$-kernel interaction with $s = d-4$, and an external field
given by a power of the Euclidean norm, namely
$\gamma\left|\cdot\right|^\alpha$, $\alpha>0$, $\gamma>0$. The covered cases
are $(d,s)\in\{(3,-1),(4,0),(5,1),\ldots\}$. 

Unlike the Coulomb case $s=d-2$, our main result (Theorem
\ref{th:main} below) reveals that the equilibrium measure can be a mixture of
a continuous part and a singular part. Furthermore, as the power $\alpha$ in
the external field increases to $2$, a transition occurs where the support of
the equilibrium measure reduces from a full $d$-dimensional ball to a $d-1$
dimensional sphere. Moreover, for powers $\alpha$ larger than $2$, the
equilibrium measure continues to be the uniform distribution on a $d-1$
dimensional sphere with an explicit special radius. In particular, this holds
for the logarithmic case $s=0$, $d=4$ and contrasts with the cases $s=d-2$ and $s=d-3$
(studied in \cite{potspe}\footnote{See also arXiv:2108.00534v6 which contains an analytic derivation of the Riesz formula for the
equilibrium measure on a ball using a special case of a result of Dyda et al~\cite{MR3640641}.}) and for which the equilibrium measures are fully
dimensional and supported on a ball for $\alpha = 2$.

It is known that a condensation phenomenon may occur for an equilibrium
measure when the Riesz parameter $s$ passes through a critical value. For
example, the equilibrium problem for the Riesz $s$-kernel on a disc in
$\mathbb{R}^2$ with no external field has support which transitions from the
full disc for $2> s > 0$ to the boundary circle for $0 \geq s > -2$, see for
instance \cite{MR78470,MR3970999,MR0350027}. In the present work, we exhibit a
new condensation phenomenon that occurs for a fixed Riesz $s$ parameter, when
the external field power passes through a critical value. Our model is
relatively simple, multivariate but radial. For further discussion of
equilibrium problems with external fields, see for instance
\cite{BalagueCarrilloLaurentRaoul2013,BilykGMPV2021,CanizoCarrilloPatacchine2015,
  CarrilloFigalliPatacchini2017,GutlebCarrilloOlver2022Balls,GutlebCarrilloOlver2021}.

\subsection{Riesz $s$-energy with an external field in $\Rd$}

For all $d\in\{1,2,\ldots\}$, and $x\in\Rd$, we write
$\ABS{x}:=(x_1^2+\cdots+x_d^2)^{1/2}$. We take $s\in(-2,+\infty)$ and for all
$x\in\Rd$, $x\neq 0$, we define the ``kernel''
\begin{equation}\label{eq:VW}
  K_s(x):=
  \begin{cases}
    \mathrm{sign}(s)\ABS{x}^{-s} & \text{if $-2<s<0$ or $s>0$}\\
    -\log\ABS{x} & \text{if $s=0$}
  \end{cases},
\end{equation}
known as the ``Riesz $s$-kernel'', and as the Coulomb or Newton kernel when
$s=d-2$. It is well known that for all integers $d\geq1$, the Coulomb kernel
$K_{d-2}$ is the fundamental solution of the Laplace or Poisson equation in
$\mathbb{R}^d$; in other words, in the sense of Schwartz distributions in
$\mathbb{R}^d$, we have $-\Delta K_{d-2}=c_d\delta_0$, where
$\Delta:=\sum_{i=1}^d\partial_i^2=\mathrm{Trace}(\mathrm{Hessian})$ is the
Laplacian and where $\delta_0$ is the Dirac unit point mass at the origin. The
constant is known explicitly, namely\footnote{\label{FN:Ks}An alternative
  non-standard definition of $K_s$ would be $K_s=1/(s\ABS{\cdot}^s)$ if
  $s\neq0$. This gives $K_0$ from $K_s$ by removing the singularity as
  $s\to0$, namely
  $\lim_{s\to0}(1/(s|x|^s)-1/s) = \lim_{s\to0}(|x|^{-s}-1)/(s-0)=-\log|x|$.
  This produces nicer formulas in general, for instance $c_d$ would be simply
  equal to $|\sph|$ for all $d\geq1$ with this choice.}
\begin{equation} \label{eq:cd}
  c_2=2\pi\quad\text{while}\quad
  c_d=
  \ds(d-2)|\sph|=(d-2)\frac{2\pi^{\frac{d}{2}}}{\Gamma\bigr(\frac{d}{2}\bigr)}
  \quad\text{if $d\neq2$}.
\end{equation}
Let $V:\Rd\to(-\infty,+\infty]$ be lower semi-continuous and such that
\begin{equation}\label{eq:Vinf}
  \inf_{x\neq y}(K_s(x-y)+V(x)+V(y))>-\infty.
\end{equation}
In this work, we focus on
\begin{equation}\label{eq:Vdef}
  V = \gamma\left|\cdot\right|^\alpha, \qquad \gamma > 0, \quad \alpha>0,
\end{equation}
and we  note that \eqref{eq:Vinf} is satisfied for $s\geq0$ and
when $|s|<\alpha$ for $s<0$.

Let $\mathcal{M}_1(\Rd)$ be the set of probability measures on $\Rd$. For all
$\mu\in\mathcal{M}_1(\Rd)$, we define
\begin{equation}\label{eq:Esalpha}
  \mathrm{I}(\mu)=\mathrm{I}_{s,V}(\mu)
  :=\iint_{\Rd\times\Rd}(K_s(x-y)+V(x)+V(y))\mu(\dd x)\mu(\dd y),
\end{equation}
the ``energy with external field $V$'' of $\mu$. Thanks to \eqref{eq:Vinf},
the integrand is bounded below and thus the double integral is well defined
but possibly infinite.

The function $\mathrm{I}$ is strictly convex\footnote{If $s\leq-2$, then we
  lose strict convexity and uniqueness, but still it is possible to
  characterize minimizers, see \cite{MR78470}.} on $\mathcal{M}_1(\Rd)$, see
\cite[Lem.~3.1]{MR3262506} and \cite[Th.~4.4.5]{MR3970999} for $0 < s < d$ and
\cite[Th.~4.4.8]{MR3970999} for $-2 < s \leq 0$. Moreover, if we equip
$\mathcal{M}_1(\Rd)$ with the topology of weak convergence with respect to
continuous and bounded test functions (weak-$*$ convergence), then
$\mathrm{I}$ is lower semi continuous with compact level sets. In particular,
it has a unique global minimizer $\meq=\mu_{s,V}\in\prob$, called the
``equilibrium measure'':
\begin{equation}\label{eq:equ}
  \mathrm{I}(\meq)= \min_{\mu\in\prob}\mathrm{I}(\mu)>-\infty
  \quad\text{and}\quad
  \mathrm{I}(\mu)>\mathrm{I}(\meq)\quad\text{for all $\mu\neq\meq$.}
\end{equation}
The condition $s>-2$ ensures conditional strict positivity of the kernel,
giving strict convexity of $\mathrm{I}$ and uniqueness of $\meq$, see
\cite{MR78470}, \cite[Sec.\ 4.4]{MR3970999}, and \cite[Ch.\ VI, p.\
363--]{MR0350027}. Note that when $s<0$, then $K_s$ is not singular, and as a
consequence we could have $\mathrm{I}(\mu)<\infty$ for a probability measure
$\mu$ with Dirac masses; in particular $\meq$ may have Dirac masses. In
contrast, when $s\geq0$ then $K_s$ is singular, and $\mathrm{I}(\mu)=+\infty$
if $\mu$ has Dirac masses; in particular $\meq$ does not have Dirac masses.

We consider hereafter only measures $\mu\in\mathcal{M}_1(\Rd)$ such that
$\int_{|x|>1} |K_s(x)| \mu(\dd x)<\infty$. Then the potential of $\mu$
is
\begin{equation}\label{eq:U}
  U^\mu(x):=U^\mu_s(x):=\int K_s(x-y)\mu(\dd y)=(K_s*\mu)(x)
  \in (-\infty,+\infty],
\end{equation}
which is finite almost everywhere in $\Rd$, see \cite[Section I.3]{MR0350027}.
The equilibrium measure $\meq$ has an Euler\,--\,Lagrange variational
characterization known as the Frostman conditions: it is the unique
probability measure for which there exists a constant $c$ such that the
modified potential
\begin{equation}\label{eq:eulerlagrange}
  U^{\meq}+V
  \begin{cases}
    =c& \text{ q.e. on $\mathrm{supp}(\meq)$}\\
    \geq c&\text{ q.e. on $\mathbb{R}^d$}.
  \end{cases}
\end{equation}
Here ``q.e.'' denotes ``quasi--everywhere'' which means except on a set for
which every probability measure supported on it has infinite energy. These
conditions hold everywhere when $V$ is continuous.

\begin{remark}[Degenerate or special cases]\label{rm:degen}
  Let $-2<s<0$, $d \geq 1$ and $V = \gamma\left|\cdot\right|^\alpha$,
  $\gamma>0$, $\alpha>0$.
  \begin{itemize}
  \item If $\alpha=-s$ and $\gamma\geq1$, then $\meq=\delta_0$ (this holds in
    particular when $\alpha=-s=\gamma=1$).
  \item If $\alpha = -s$ and $\gamma < 1$ or if $\alpha < -s$, then $\meq$
    does not exist (and \eqref{eq:Vinf} is not satisfied).
  \item If $\alpha>-s$, then $\meq\neq\delta_0$ (and \eqref{eq:Vinf} is satisfied).
  \end{itemize}
  Let us give a brief proof of these statements. We first observe that, for
  $x\in\mathbb{R}^d$, the modified potential is
  $U^{\delta_0}(x) + V(x) = -|x|^{-s} + \gamma |x|^\alpha$, which for
  $\lambda:=|x|$ becomes
  $\varphi(\lambda):=-\lambda^{-s}+\gamma\lambda^\alpha$. We now use the
  Frostman conditions \eqref{eq:eulerlagrange} to treat, in turn, the bullet
  points above. If $\alpha=-s$ then,
  $\varphi(\lambda)=(\gamma-1)\lambda^\alpha$ and the Frostman conditions
  \eqref{eq:eulerlagrange} hold when $\gamma\geq1$. If $\alpha = -s$ and
  $\gamma < 1$ or if $\alpha<-s$, then
  $\lim_{\lambda\to\infty}\varphi(\lambda)=-\infty$ and the Frostman
  conditions \eqref{eq:eulerlagrange} cannot hold. Finally, when $\alpha>-s$,
  $\varphi(0)=0$, and $\varphi'(\lambda)=0$ has the unique solution
  $\lambda_*=(\frac{-s}{\gamma\alpha})^{\frac{1}{\alpha+s}}>0$. Now
  $\varphi(\lambda)<0$ if and only if
  $\lambda<(\frac{1}{\gamma})^{\frac{1}{\alpha+s}}$, so
  $\varphi(\lambda_*) < 0$, which contradicts the Frostman conditions
  \eqref{eq:eulerlagrange}; hence $\meq\neq\delta_0$.
\end{remark}

\subsection{Threshold phenomena for $s = d-4$}

Our main result below reveals threshold phenomena, when $d=s+4$, at
$\alpha=2$, $\alpha=1$, and $\gamma=1$. We use the following notation:
\begin{itemize}
\item $\sph_R:=\{x\in\mathbb{R}^d:|x|=R\}$, sphere of radius $R$ in
  $\mathbb{R}^d$ centered at the origin, and $\sph:=\sph_1$;
\item $\sigma_{R}$ : uniform probability measure on $\sph_R$;
\item $m_d$ : Lebesgue measure on $\mathbb{R}^d$.
\end{itemize}

\begin{theorem}[Main result]\label{th:main}
  Let $V=\gamma \ABS{\cdot}^\alpha$ where $\gamma>0$ and $\alpha>0$.
  \begin{enumerate}[label=(\roman*)]
  \item \label{it:thma}
  Suppose that $d \geq 4$ and $s = d-4 \geq 0$.
    \begin{enumerate}[label=(\alph*)] \label{it:ball}
    \item If $0 < \alpha < 2$, then 
      \begin{equation}\label{eq:mixture}
        \meq=\beta f m_d+(1-\beta)\sigma_{R},
      \end{equation}
      where
      \begin{equation}\label{eq:betaf}
        \beta:=\frac{2-\alpha}{s+2},\quad
        f(x):=\frac{\alpha+s}{R^{\alpha+s}|\sph|}|x|^{\alpha-4}\mathbf{1}_{|x|\leq
          R}
      \end{equation}
      and
      \begin{equation}\label{eq:Rball}
        R:=
        \begin{cases}
          \PAR{\frac{2|s|}{(\alpha+s+2)\gamma\alpha}}^{\frac{1}{\alpha+s}}
          &\text{if $s\neq0$}\\
          \PAR{\frac{2}{(\alpha+2)\gamma\alpha}}^{\frac{1}{\alpha}}
          &\text{if $s=0$}
        \end{cases}.
      \end{equation}
    \item \label{it:sphere} If $\alpha\geq2$, then $\meq=\sigma_{R}$
      where
      \begin{equation}\label{eq:Rsph}
        R:=
        \begin{cases}
          \PAR{\frac{2|s|}{(s+4)\gamma\alpha}}^{\frac{1}{\alpha+s}}
          &\text{if $s\neq0$}\\
          \PAR{\frac{1}{2\gamma\alpha}}^{\frac{1}{\alpha}}
          &\text{if $s=0$}
        \end{cases}.
      \end{equation}
      Moreover, when $s=0$ this remains the equilibrium measure for all $d\geq4$.
    \end{enumerate}
  \item \label{it:thmb}
  Suppose that $d = 3$ and $s = d - 4 = -1$.
    \begin{enumerate}[label=(\alph*)]
    \item \label{it:pt3}
      If $\alpha = 1$, and $\gamma\geq 1$, then $\meq=\delta_0$
      (this holds in particular when $\alpha=\gamma=1$).
    \item \label{it:ball3}
      If $1 < \alpha < 2$, then $\meq$ is the mixture given by \eqref{eq:mixture}, \eqref{eq:betaf} and \eqref{eq:Rball}.
    \item \label{it:sph3}
      If $\alpha \geq 2$, then $\meq=\sigma_{R}$ with $R$ given by \eqref{eq:Rsph}.
    \end{enumerate}
  \end{enumerate}
\end{theorem}

Theorem \ref{th:main} is proved in Section \ref{se:th:main:proof}.

Let us give some observations about Theorem \ref{th:main}:
\begin{enumerate}[(i)]
\item If $d=3$, $s=-1$, $\alpha = 1$ and $0<\gamma < 1$ or $0<\alpha<1$, then
  $\meq$ does not exist and \eqref{eq:Vinf} fails.
\item The critical radius in the case $\alpha\geq2$ is also the critical
  radius for the equilibrium problem restricted to spheres, see Lemma
  \ref{le:spheres:opt}.
\item A convex combination of probability measures as in \eqref{eq:mixture} is
  known as a ``mixture''. More precisely \eqref{eq:mixture} is a mixture of
  the absolutely continuous probability measure $fm_d$ and the singular
  probability measure $\sigma_R$. Note that $fm_d$ is itself a mixture, since
  it is the law of the product $VU$ where $U$ and $V$ are independent random
  variables with $U$ uniform on the unit sphere of $\mathbb{R}^d$ and $V$
  supported in $[0,R]$ with density
  $r\mapsto\frac{(\alpha+s)}{R^{\alpha+s}}r^{\alpha+s-1}$.\label{it:itemradial}
\item If $s\to0$, we do not recover the case $s=0$, and $R$ is discontinuous
  at $s=0$. This is due to our choice of normalization with respect to $s$ of
  $K_s$, see Footnote~\ref{FN:Ks}.
\item Theorem \ref{th:main} is in accordance with the numerical experiments
  depicted in Figure \ref{fi:rad}. Note that the case $d=4$, $s=0$, the range
  $0<\alpha<1$ is less reliable numerically than the range $\alpha\geq1$,
  since in this case the radial density provided in item \ref{it:itemradial} becomes
  singular at the origin.
\item When $d=3$ and $s=-1$, the interaction is not singular at the origin,
  but is singular at infinity, producing long range interactions in the
  energy.
\item If $d=3=3+0$, $s=0$, and $\alpha\geq2$, then $\meq$ is no longer
  supported on a sphere, but rather on a $3$-dimensional ball, see
  \cite{potspe}.
\end{enumerate}

\begin{remark}[Behavior of $\meq$ with respect to $\alpha$ in Theorem
  \ref{th:main}] \label{re:cont} The equilibrium measure $\meq$ in Theorem
  \ref{th:main} is ``continuous'' with respect to $\alpha$ in the following
  sense (see Figures \ref{fi:rad} and \ref{fi:hist}):
  \begin{itemize}
  \item If $\alpha\to\infty$, then from \eqref{eq:Rsph}, $R \to 1^-$.
  \item If $\alpha\to 2^-$, then the continuous part $\beta f m_d$, where
    $\beta = (2 - \alpha)/(s+2)$, of $\meq$ in \eqref{eq:mixture} vanishes and
    we recover the formula for $\alpha=2$.
  \item If $\alpha\to1^+$ and $\gamma\geq1$, then $R\to0$ and we recover
    the fact that $\meq=\delta_0$ when $\alpha=1$ and $\gamma\geq1$.
  \item If $\alpha\to1^+$ and if $0<\gamma<1$, then $R\to\infty$ and we
    recover the fact that $\meq$ does not exist when $\alpha=1$ and
    $0<\gamma<1$.
  \item As $\alpha\to 0^+$, then from \eqref{eq:Rball}, $R\to\infty$.
  \end{itemize}
\end{remark}

We remark that in the case $-2<s<0$ ($K_s$ is not singular) and $\alpha=2$, the equilibrium problem 
\[
  \arg\min_\mu\Bigr\{ 
  -\iint|x-y|^{|s|}\mu(\mathrm{d} x)\mu(\mathrm{d} y)
  +2\int\left|x\right|^2\mu(\mathrm{d} x)\Bigr\}  
  \]
arises in steepest descent for halftoning functionals, see \cite{hertrich-graf-beinert-steidl}, and there are explicit formulas for $\mu_{\mathrm{eq}}$.

\subsection{Numerical experiments for discrete energy}
\label{S:NumExp}

  It is natural to approximate a probability measure $\mu$ on $\dR^d$ by an
  empirical measure
  $\mu_{x_1,\ldots,x_N}:=\frac{1}{N}\sum_{i=1}^N\delta_{x_i}$, for a well
  chosen configuration $x_1,\ldots,x_N$ of $N$ points in $\dR^d$, and its
  energy $\mathrm{I}(\mu)$ by
  \begin{align}\label{eq:En}
    E(\mu_{x_1,\ldots,x_N}) &:=
    \sum_{i = 1}^N \sum_{\substack{j = 1\\j\neq i}}^N
    \left(K_s(x_i-x_j) + V(x_i) + V(x_j)\right)\\
    &=
    2\Bigr( \sum_{i = 1}^N \sum_{j = i+1}^{N} K_s(x_i-x_j) +
     (N-1)\sum_{i=1}^N V(x_i)\Bigr).
  \end{align}
  The removal of the diagonal ensures that \eqref{eq:En} 
  is finite as soon as
  $x_1,\ldots,x_N$ are distinct, despite the singularity at the origin of
  $K_s$ when $s\geq0$. Actually for $s < 0$, we do not have to remove the
  diagonal and we can sum over all $i, j$, with no contributions from the kernel
  for $i = j$ and $2 N \sum_{i=1}^N V(x_i)$ for the external field.

 \begin{figure}[htbp]
  \centering
  \includegraphics[width=.60\textwidth,trim=2.5cm 8.1cm 2.5cm 8.2cm]{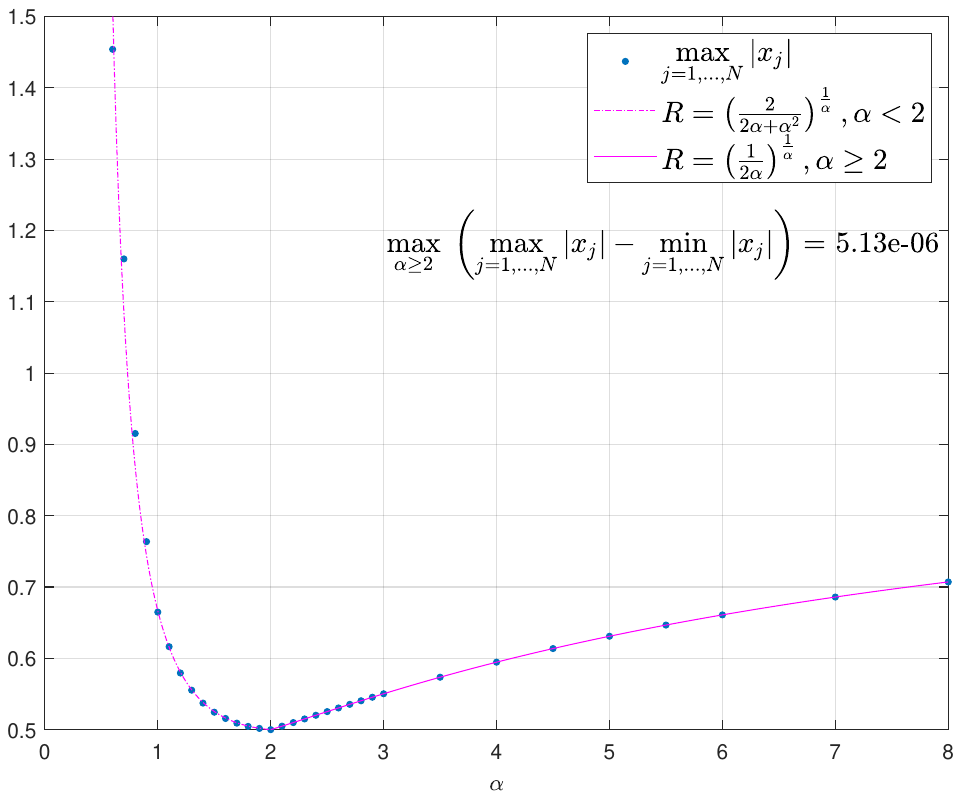}
  \caption{\label{fi:rad} Support radius $R$ of $\meq$ and
    $\displaystyle\max_{j=1,\ldots,N} |x_j|$ for empirical measure with
    $N = 10^4$ points, when $d=4$, $s=d-4=0$, and $V=\gamma\ABS{\cdot}^\alpha$
    with $\gamma=1$ and $\alpha>0$.}
\end{figure}
A local minimum of a smooth unconstrained function can be found efficiently by a
variety of gradient based descent methods, see~\cite{NocedalWright2006} for
example. Here
\begin{equation}\label{eq:GradEn}
  \nabla_{x_k} E(\mu_{x_1,\ldots,x_N}) =
    \sum_{\substack{j = 1\\j\neq k}}^N \nabla_{x_k} K_s(x_k - x_j) +
    2(N-1) \nabla_{x_k} V(x_k)
\end{equation}
with $\nabla_{x_k} V(x_k) = \alpha \gamma |x_k|^{\alpha-2}\; x_k$, and for $k \neq j$,
\begin{equation}\label{eq:GradK}
   \nabla_{x_k} K_s(x_k - x_j) =
    \begin{cases}
    |s| ^2\frac{x_k - x_j}{|x_k - x_j|^{s+2}} & \text{if $-2<s<0$ or $s>0$}\\
    -\frac{x_k-x_j}{|x_k - x_j|^2} & \text{if $s=0$}
  \end{cases}.
\end{equation}

\begin{figure}[htbp]
  \centering
  \includegraphics[width=.85\textwidth,trim=2.5cm 8.5cm 2.5cm 8cm]{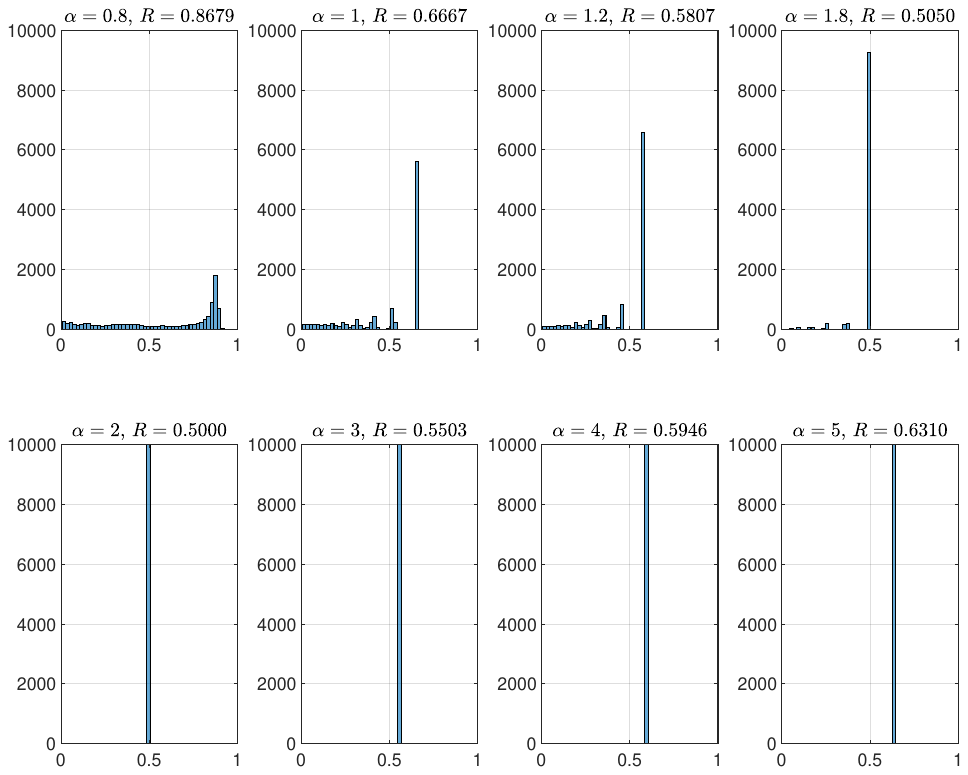}
  \caption{\label{fi:hist} Histograms of point norms
    $|x_j|, j = 1,\ldots,N=10^4$ of empirical equilibrium measure for $d=4$,
    $s=d-4=0$, and $V=\gamma\ABS{\cdot}^\alpha$ with $\gamma=1$ and selected
    $\alpha>0$.}
\end{figure}

The number of optimization variables is $n = dN$ for $N$ points in $\dR^d$,
for example $n = 40,000$ for the experiments in Figures~\ref{fi:rad} and
\ref{fi:p}, so a limited-memory BFGS
method~\cite{ByrdLuNocedalZhu1995} can avoid computations with an $n$ by $n$
Hessian approximation. Gradient information is essential due to the highly
nonlinear interactions from the Riesz kernels. 
For each fixed $N$, there can be many local minima to the discrete optimization problem, thus many different initial points were used to try to identify the global minimum.
Furthermore, the external field $V = \gamma\left|\cdot\right|_p^\alpha$ has a
derivative discontinuity at the origin for $p = 2$, $\alpha < 2$, making
discrete optimization for densities with a mass at the origin much more
difficult. Also for $p = 1$, the external field has derivative discontinuities
whenever a component of one of the points is zero.

We further remark that as $N \to \infty$ the sequence of empirical measures
$\mu_N^* := \mu_{x_1^*,\ldots,x_N^*}$ for minimizers
$\omega_N^* = (x_1^*,\ldots,x_N^*)$ of the discrete energy in \eqref{eq:En}
converges in the weak-star topology to the corresponding equilibrium measure
$\meq$ of Theorem~\ref{th:main}. Indeed, it is easy to show that thanks to the
growth of $V(x)$ at infinity, the empirical measures are all supported on a
compact set independent of $N$ and so a standard argument (see for instance
Theorem 4.2.2 of \cite{MR3970999}) going back to Choquet~\cite{Choquet1958}
and Fekete~\cite{Fekete1923} shows that any limit measure of the sequence
$\mu_N^*$ is necessarily an equilibrium (minimizing) measure of the continuous
problem in \eqref{eq:equ}. Moreover, under the assumptions of
Theorem~\ref{th:main}, this equilibrium measure is unique.

Figures~\ref{fi:rad} and \ref{fi:hist} illustrate the results of some
numerical experiments minimizing the modified potential \eqref{eq:En} using
$N = 10^4$ discrete points $x_j, j = 1,\ldots,N$ for $d = 4$, $s = 0$,
$V = \gamma\left|\cdot\right|^\alpha$, $\gamma = 1$ and various values for
$\alpha$. Figure~\ref{fi:rad} shows the strong agreement between the empirical
support radius $\max_{j=1,\ldots,N} |x_j|$ of the discrete measure and the
theoretical results in Theorem~\ref{th:main}. Figure~\ref{fi:rad} also
illustrates the continuity of the support radius at $\alpha = 2$, as discussed
in Remark~\ref{re:cont}. Figure~\ref{fi:hist} gives histograms of the discrete
measure for various $\alpha$, illustrating the change from $\alpha < 2$ when
the equilibrium measure is the mixture \eqref{eq:mixture} to $\alpha \geq 2$
when the equilibrium measure is the uniform probability density on a sphere of
radius $R$ given by \eqref{eq:Rsph}.

\section{Remarks and conjectures for general $s$ and $d$}

\subsection{More general values of $s$ and $d$}

We could ask about $\meq$ for more general values of $s$ and $d$. Here are
some remarks about a few other cases, beyond the case $d=s+4$ of Theorem
\ref{th:main} and the case $d=s+3$ and $\alpha=2$ of \cite{potspe}.

\subsubsection{Case $s=d-1$}\label{ss:s=d-1}

To the best of our knowledge, little is known when $s=d-1$. However, if
$d=1$, $s=d-1=0$ , $V=\ABS{\cdot}^\alpha$, $\alpha>0$, and following
\cite[Th.~IV.5.1]{MR1485778} (see \cite[Pro.~5.3.4]{MR1746976} for relation
to free probability), then the equilibrium measure $\meq$ is the Ullman
distribution on $\mathbb{R}$ with density
\begin{equation}\label{eq:ullman}
  x\in\mathbb{R}\mapsto\frac{\alpha}{\pi R^\alpha}
  \Bigr(\int_{|x|}^{R}\frac{t^{\alpha-1}}{\sqrt{t^2-x^2}}\dd t\Bigr)
  \mathbf{1}_{x\in[-R,R]}
  \quad\text{where}\quad
  R:=\PAR{\frac{\sqrt{\pi}}{\alpha}\frac{\Gamma\bigr(\frac{\alpha+2}{2}\bigr)}{\Gamma\bigr(\frac{\alpha+1}{2}\bigr)}}^{\frac{1}{\alpha}}.
\end{equation}
When $\alpha=2$ we recover a semicircle distribution with density
$\frac{2}{\pi}\sqrt{1-x^2}\mathbf{1}_{x\in[-1,1]}$. Note also that the formula
for the radius above has a form similar to the critical radius in the case
$s=d-3$ in \cite{potspe}.

When $d-2 < s < d$, so that the maximum principle
holds, then, one can minimize the
Riesz analogue of the Mhaskar\,--\,Saff functional (see \cite[Chap.~IV,~eq.~(1.1)]{MR1485778})
\[
  \mathcal{F}_s(A):=\iint K_s(x-y)\mu_A(\dd x)\mu_A(\dd y)+\int V(x)\mu_A(\dd x),
\]
where $A$ is a compact set of positive capacity
and $\mu_A$ is the equilibrium measure for the Riesz minimum $s$-energy problem for $A$ with no external field.
Assuming that the support $A$ of $\meq$ is a ball $B_R$, for which $\mu_{B_R}$ is known, then minimizing $\mathcal{F}_s(B_R)$ leads to the following formula for the radius of the ball
\begin{equation}\label{eq:BallR}
  R = \left(\frac{|s|\Gamma\bigr(\frac{d - s}{2}\bigr)\Gamma\bigr(\frac{\alpha + 2+s}{2}\bigr)}{\alpha\Gamma\bigr(\frac{\alpha + d}{2}\bigr)}\right)^\frac{1}{\alpha + s},\,\,\,\,\alpha+s>0.
\end{equation}

To verify this formula one can use the Frostman conditions \eqref{eq:eulerlagrange}. For $s = d-1$, $d \geq 2$ and $\alpha = 2$ the radially symmetric semicircle distribution on the ball $B_R$ is,
with $y = r R \yh$, $r \in [0, 1]$, $\yh\in \sph$,
\[
   \mu_R(\dd y) =
   M \sqrt{R^2 - |y|^2}\  \mathbf{1}_{|y|\leq R}\; \dd y =
   M R^{d+1} \sqrt{1 - r^2} \; r^{d-1} \mathbf{1}_{r\in [0, 1]}\,
          \dd r \, \sigma_1(\dd \yh),
\]
where the normalization constant $M$ satisfies
\[
  M^{-1} = R^{d+1} \int_0^1 \sqrt{1-r^2} \, r^{d-1} \dd r
         = R^{d+1} \frac{\sqrt{\pi}}{4}
           \frac{\Gamma\left(\frac{d}{2}\right)}{\Gamma\left(\frac{d+3}{2}\right)}
           = \frac{d-1}{d} \frac{\pi}{4} .
\]
Also using the parametrization $x = \lambda R \xh$, $\lambda \geq 0$, $\xh\in \sph$,
and the Funk\,--\,Hecke formula of Lemma~\ref{le:funkhecke} for the integral of the Riesz kernel over $\sph$, gives
\begin{eqnarray*}\label{eq:sd1modpot}
 U^{\mu_R}(x) & = & 
  \int_{B_R} K_s(|x - y|) \mu_R(\dd y)  
   =  R^{d-s+1} M \int_0^1 \int_{\sph} |\lambda \xh - r \yh|^{-s}  \,
 r^{d-1} \sqrt{1 - r^2} \; \sigma_1(\dd \yh) \, \dd r \\
 & = & \frac{R^2 4d}{(d-1)\pi} \int_0^1 (\lambda + r)^{1-d}
    {}_2 F_1\left(\frac{d-1}{2}, \frac{d-1}{2}; d-1;
    \frac{4\lambda r}{(\lambda + r)^2}\right) \,
    r^{d-1} \sqrt{1 - r^2} \, \dd r .
\end{eqnarray*}
The Frostman conditions
for $\mu_R$ to be the equilibrium measure
are that the modified potential $\varphi(\lambda) := U^{\mu_R}(\lambda R \xh) + (\lambda R)^2$
satisfies $\varphi(\lambda) = c$ for $\lambda \in [0, 1]$ and
$\varphi(\lambda) > c$ for $\lambda > 1$.
Here $c = R^2 \frac{d}{d-1}$ and $\varphi(\lambda) = c$ is equivalent to,
for $\lambda \in [0,1]$,
\begin{equation}\label{eq:sd1id}
  \int_0^1 (\lambda + r)^{1-d}
   {}_2 F_1\left(\frac{d-1}{2}, \frac{d-1}{2}; d-1;
    \frac{4\lambda r}{(\lambda + r)^2}\right)
    r^{d-1} \sqrt{1 - r^2} \; \dd r =
    \frac{\pi}{4} \left( 1 - \frac{d-1}{d} \lambda^2\right),
\end{equation}
which is \cite[Lem.~2.4 with $\alpha=1-d$ and $\ell=1$]{GutlebCarrilloOlver2022Balls} divided by $R^2 |\sph|$.
Moreover, numerical experiments suggest that the inequality in the Frostman conditions \eqref{eq:eulerlagrange} holds when $\lambda>1$. It is worth noting that when $d=2$, then \eqref{eq:sd1id} boils down to the following formula in the same spirit of those in \cite{potspe}:
\begin{equation}\label{eq:sd1id2}
  \int_0^1
   K\left(\frac{4\lambda r}{(\lambda + r)^2}\right)
    \frac{r\sqrt{1 - r^2}}{\lambda + r}\; \dd r =
    \frac{\pi^2}{16} \left( 2 - \lambda^2\right),
\end{equation}
where $K(z)=\dfrac{\pi}{2}\, {}_2F_1(1/2,1/2;1;z)$ is the complete Elliptic integral of the first kind, (cf.\ \cite[Eq.~(1.20)]{potspe}).

\subsubsection{Case $s=d-2$ (Coulomb)}\label{ss:s=d-2}

Let us consider the Coulomb case $s=d-2$. From
\begin{equation}\label{eq:Delta}
  \Delta K_{d-2}=-c_d\delta_0
\end{equation}
we get the inversion formula, in the sense of distributions,
\begin{equation}\label{eq:inv}
  \Delta U^\mu=-c_d\mu.
\end{equation}
When $V$ is locally integrable it can be viewed as a distribution and we
get, by combining \eqref{eq:eulerlagrange} and \eqref{eq:inv}, that $\meq$
is equal on the interior of its support, in the sense of distributions, to
the distribution
\begin{equation}\label{eq:DeltaV}
  \frac{\Delta V}{c_d}.
\end{equation}
Beware that $\meq$ is not necessarily absolutely continuous and may have a
singular part outside the interior of its support. In particular, when $V$
is $\mathcal{C}^2$ then the interior of the support of $\meq$ does not
intersect the set $\{\Delta V<0\}$.

From \cite[Prop.~2.13]{MR2647570}, when $d \geq 3$,
$V = \left|\cdot\right|^\alpha$, $\alpha > 0$ (so $\gamma = 1$) the support
of the equilibrium measure $\meq$ is a ball of radius $R$, where
\begin{equation}\label{eq:s=d-2R}
  R = \left( \frac{d-2}{\alpha} \right)^\frac{1}{d+\alpha-2}
\end{equation}
and
\begin{equation}\label{eq:s=d-2meq}
  \dd\meq(r \xh) = \frac{\alpha (d+\alpha-2)}{d-2}\; r^{d+\alpha-3}\;
  \ind_{0\leq r\leq R} \; \dd r\; \sigma_1(\dd\xh) .
\end{equation}
When $d = 2$, so $s = 0$, \cite[Thm.~IV.6.1]{MR1485778} shows that
\begin{equation}\label{eq:d=2s=0R}
  R = \left( \frac{1}{\alpha} \right)^\frac{1}{\alpha}
\end{equation}
and
\begin{equation}\label{eq:d=2s=0meq}
  \dd\meq(r \xh) = \alpha^2
  r^{\alpha-1}\; \ind_{0\leq r\leq R}\;
  \dd r\; \sigma_1(\dd\xh) .
\end{equation}
Note that when $\alpha = 2$, $\meq$ is uniform on the ball in $\Rd$ with
volume $c_d/(2d)$, which has radius $R$.

\subsubsection{Case $s = d - 3$}

It is proved in \cite{potspe} that when $d=s+3$ and
$V=\gamma\left|\cdot\right|^2$, $\gamma>0$, then $\meq$ has a radial arcsine
distribution supported on a ball. It is also mentioned in \cite{potspe} that
an explicit computation of $\meq$ is still possible when $V$ belongs to a
special class of radial polynomials or hypergeometric functions. On the other
hand, when $d=3$, $s=0$, $V=\gamma\left|\cdot\right|^\alpha$, $\gamma>0$, it
is easily proved that for no value of $\alpha>0$ is the support of $\meq$ a
single centered sphere; indeed, it is easy to show that this would violate the
Frostman conditions. It is tempting to conjecture that for $d=s+3$, $s\geq0$,
and for any $\alpha>0$, the support of $\meq$ has full dimension. Numerical
experiments suggest that when $d=s+3$ then the support of $\meq$ could be a
ball when $\alpha \leq 2$ and a shell (region between two concentric spheres)
when $\alpha > 2$.

\subsubsection{Case $s=d-2n$, $n\in\{1,2,3,\ldots\}$, iterated
  Coulomb}\label{ss:s=d-2n}

We have the following proposition:

\begin{proposition}[Iterated Coulomb]\label{pr:s=d-2n}
  Suppose that $V$ is locally integrable and is such that $\meq$ exists and is characterized by the Frostman conditions \eqref{eq:eulerlagrange}.
    Then, for $s=d-2n$, the equilibrium measure $\meq$ is equal, on the interior of
  its support, in the sense of distributions, to the distribution
  \begin{equation}\label{eq:s=d-2n}
    \frac{\Delta^{n}V}{c_dC_{d,n}}
  \end{equation}
  where $c_d$ is defined in \eqref{eq:cd}, 
  \begin{equation}
    C_{d,n}:=(-1)^{n-1}\frac{(d-4)!!(2n-2)!!}{(d-2n-2)!!},
  \end{equation}
  and $z!!:=\prod_{k=0}^{\lceil\frac{z}{2}\rceil-1}(z-2k)$ is the \emph{double
	factorial} (with $z!!:=1$ if $z\leq0$).\\ In
  particular, if $V$ is $\mathcal{C}^{2n}$ on an open set $O\subset\mathbb{R}^d$, then
  \begin{itemize}
  \item 
    $O\cap\mathrm{int}(\mathrm{supp}(\meq))\cap\{\Delta^{n}V<0\}=\varnothing$ when $n$ is odd (since then $C_{d,n}>0$);
  \item $O\cap\mathrm{int}(\mathrm{supp}(\meq))\cap\{\Delta^{n}V>0\}=\varnothing$ when $n$ is even (since then $C_{d,n}<0$).
  \end{itemize}
\end{proposition}

Before we give a proof of Proposition \ref{pr:s=d-2n}, we make the
following observations:
\begin{itemize}
\item $2n=d-s\leq d+1$ since $s>-2$
\item Proposition \ref{pr:s=d-2n} with $n=1$ is the Coulomb case $s=d-2$ of
  Section \ref{ss:s=d-2}
\item $\Delta\ABS{\cdot}^\alpha=\alpha(\alpha+d-2)\ABS{\cdot}^{\alpha-2}$
  which has the same sign as $d+\alpha-2$ (critical value is $\alpha=2-d$)
\item Beware that $\meq$ is not necessarily absolutely continuous and can have
  a singular part supported outside the interior of its support, as shown in
  Theorem \ref{th:main}. Indeed, let us consider the case $d=s+4=s+2n\geq4$,
  $n=2$, and $V=\gamma\ABS{\cdot}^\alpha$, $\gamma>0$. If $0<\alpha<2,$ then
  Theorem \ref{th:main} states that $\meq$ is equal, in the sense of
  distribution, on the interior of its support, to
  \[
    \beta f
    =
    \frac{2-\alpha}{d-2}
    \frac{\alpha+d-4}{|\mathbb{S}^{d-1}|}\ABS{\cdot}^{\alpha-4}
    \gamma\frac{\alpha(\alpha+d-2)}{2}
    \begin{cases}
      \frac{1}{d-4} & \text{if $d\neq4$}\\
      1 & \text{if $d=4$}
    \end{cases} .
  \]
  Alternatively, for $d \geq 3$, $n = 2$ (noting that
  $C_{4,2} = -2, C_{3,2} = 2$)
  \[
    \frac{\Delta^nV}{c_dC_{d,n}}
    =\gamma\frac{\Delta^2\ABS{\cdot}^\alpha}{c_dC_{d,2}}
    =-\gamma\frac{\alpha(\alpha+d-2)(\alpha-2)(\alpha+d-4)}
                 {2(d-2)|\mathbb{S}^{d-1}|}
      \ABS{\cdot}^{\alpha-4}
    \begin{cases}
      \frac{1}{d-4} & \text{if $d \neq 4$}\\
      1 & \text{if $d = 4$}
    \end{cases}
  \]
  which matches the formula for $\beta f$ (including the case $d = 3$,
  $s = -1$ for $1 < \alpha < 2$)!

  In contrast, if $\alpha\geq2,$ then
  \[
    \Delta^2V=\gamma\alpha(\alpha+d-2)(\alpha-2)(\alpha+d-4)\ABS{\cdot}^{\alpha-4}>0,
  \]
  and since $C_{d,2}<0$ we get $\{\Delta^2V\geq0\}=\varnothing$ therefore
  $\mathrm{int}(\mathrm{supp}(\meq))=\varnothing$, which implies that $\meq$
  is singular. Indeed Theorem \ref{th:main} gives in that case
  $\meq=\sigma_R$.

\end{itemize}

\begin{proof}[Proof of Proposition \ref{pr:s=d-2n}] For $u>-2$ with
  $u\neq d-2$, in the sense of distributions,
  \begin{equation}\label{eq:DeltaK}
    \Delta K_u=-c_{d,u}K_{u+2},
    \quad\text{where}\quad
    c_{d,u}=
    \begin{cases}
      |u|(d-2-u) & \text{if $u\neq0$}\\
      d-2& \text{if $u=0$}
    \end{cases}.
  \end{equation}
  The idea is to apply $\Delta$ repeatedly, $n-1$ times, to pass, via
  \eqref{eq:DeltaK}, by $+2$ steps, from $K_s$ to $K_{d-2}$, and then to use
  \eqref{eq:Delta}. We know that $\meq$ exists and is unique and satisfies the
  Euler\,--\,Lagrange equations \eqref{eq:eulerlagrange}. Applying $\Delta^n$
  to both sides of \eqref{eq:eulerlagrange} gives
  (see \cite[Lemma A.1 (iii)]{potspe})
  \begin{equation}\label{eq:dmeulerlagrange}
    \Delta^n K_s * \meq = -\Delta^n V
    \quad\text{on $\mathrm{supp}(\meq)$.}
  \end{equation}
  If $n>1$, iterating \eqref{eq:DeltaK} and using \eqref{eq:Delta}
  we get, with $C_{d,n}=(-1)^{n-1}\prod_{k=0}^{n-2}c_{d,s+2k}$,
  \begin{equation}\label{eq:deltam}
    \Delta^n K_s
    =\Delta(\Delta^{n-1} K_s)
    =C_{d,n}\Delta K_{s+2(n-1)}
    =C_{d,n}\Delta K_{d-2}
    =-C_{d,n}c_d\delta_0.
  \end{equation}
  Recall that $s=d-2n$. Moreover $C_{d,n}>0$ if $n$ is odd while $C_{d,n}<0$
  if $n$ is even. The formula \eqref{eq:deltam} remains valid for $n=1$
  (namely the Coulomb case $d=s+2$) by taking $C_{d,1}=1$, and reduces then to
  \eqref{eq:Delta}. By combining \eqref{eq:dmeulerlagrange} and
  \eqref{eq:deltam}, we get that $\meq$ is equal to
  $(C_{d,n}c_d)^{-1}\Delta^nV$ on the interior of $\mathrm{supp}(\meq)$. Let
  us compute the value of $C_{d,n}$ for $n \geq 2$. If $s=d-2n >-2,$ then
  $d>2$ and
  \begin{align*}
    (-1)^{n-1}C_{d,n}
    & =  \prod_{k=0}^{n-2}c_{d,s+2k} = \prod_{k=0}^{n-2}(s+2k)(d-2-s-2k) \\
    & = \frac{(s+2n-4)!!}{(s-2)!!} \  \frac{(d-s-2)!!}{(d-s-2n)!!} \\
    & = \frac{(d-4)!!}{(d-2n-2)!!} \  (2n-2)!!. 
  \end{align*}
  Note that this formula also gives $C_{d,1}=1$ for the Coulomb case $n=1$, as
  desired.
\end{proof}

\subsection{Different norms}

Numerical experiments using different norms, for example
$V=\left|\cdot\right|_p^p$, where
$\left|x\right|_p^p:=|x_1|^p+\cdots+|x_d|^p$, give intriguing results, in
accordance with \eqref{eq:s=d-2n}. See for instance Figure~\ref{fi:p}
for $p=4$. In this case the support of $\meq$ has the symmetries of the
$\ell^p$ norm but is still a mystery and we do not know if $\meq$ is
absolutely continuous, singular, or a mixture of both. Note that the uniform
distribution on $\ell^p$ balls and spheres in $\mathbb{R}^d$ admits remarkable
representations and characterizations, see for instance
\cite{MR2123199,MR2498715}. Another possibility is to modify the kernel,
namely to take $\ABS{\cdot}_p^{-s}$, and the first question is then the
positive definiteness in order to get convexity and uniqueness of $\meq$.

In $\mathbb{R}^d$ with $d=2n$, when $V=\ABS{\cdot}_{2n}^{2n}$, the support of $\meq$ has the symmetries of $\ABS{\cdot}_{2n}$ (spherical when $n=1$). Furthermore, since $\Delta^{n}V=(2n)! d$ is a non-zero constant, we get from \eqref{eq:s=d-2n} that if $n$ is odd, then $\meq$ is uniform (constant density) on the interior of its support, while if $n$ is even then the interior of the support of $\meq$ is empty and $\meq$ is singular. If $n=1$, then $\Delta^{n}V$ is constant (Coulomb case) and $\meq$ is the uniform law on a ball.

\section{Proof of Theorem \ref{th:main}}
\label{se:th:main:proof}

We first consider the equilibrium problem restricted to spheres (Section
\ref{se:spheres}). This spherical case is essential for the proof of the case
$\alpha\geq2$ which is given in Section \ref{se:alpha>=2}. We then provide the
proof of the case $d\geq4$ and $1<\alpha<2$ (Section \ref{se:alpha02}), and
then the proof of the case $d=3$ (Section \ref{se:d=3etc}). As in the
statement of Theorem \ref{th:main}, we have to consider separately the cases
$d=3$ ($s=-1$), $d=4$ ($s=0$), and $d>4$ ($s>0$), which is done
respectively in Lemmas \ref{le:phi':d=3s=-1a>=2}, \ref{le:phiop}, and
\ref{le:phi':d>=5s=d-4a>=2}.

\subsection{Optimal spheres}\label{se:spheres}

We first consider the equilibrium problem restricted to spheres and provide
some simple lemmas that are needed in the proof of part~\ref{it:sphere} of
assertion ($i$) of Theorem~\ref{th:main}.

Throughout this section, $V=\gamma\ABS{\cdot}^\alpha$, $\gamma>0$, $\alpha>0$,
and we use of the fact that for $R > 0$, $K_s*\sigma_R + V$ is
radially symmetric, and so for $x\in\Rd$, we introduce the parametrization $x=\lambda R \xh$ with
$\lambda\geq0$ and $|\xh|=1$, and we only need to consider
\begin{equation}\label{eq:sphere:phi}
  \varphi(\lambda)
  =\varphi_{\sigma_R}(\lambda)
  :=(K_s*\sigma_R+V)(\lambda R \xh),
  \quad \lambda \geq 0, \quad \xh\in\mathbb{S}^{d-1}.
\end{equation}
Note that with this parametrization, the critical value for $\lambda$ is $\lambda=1$, regardless of the value of $R$.

\begin{lemma}
  \label{le:spheres}
  Let $d\geq2$, $s>-2$ and $\varphi$ be as in \eqref{eq:sphere:phi}. Then, with the above notation,
  for $x\in\mathbb{R}^d$, and with $\tau_{d-1}$ as in Lemma
  \ref{le:funkhecke}, we have
  \[
    \varphi(\lambda)
    =
    \begin{cases}
      \ds \frac{\tau_{d-1}}{R^s}
      \int_{-1}^1
      \frac{(1-t^2)^{\frac{d-3}{2}}}{(\lambda^2+1-2\lambda t)^{\frac{s}{2}}}\dd t
      +\gamma(\lambda R)^{\alpha}
      & \text{if $s\neq0$}\\[1em]
      \ds -\log R
      -\frac{\tau_{d-1}}{2}
      \int_{-1}^1\log(\lambda^2+1-2\lambda t)(1-t^2)^{\frac{d-3}{2}}\dd t
      +\gamma(\lambda R)^\alpha
      & \text{if $s=0$}
    \end{cases}.
  \]
  Moreover $\varphi$ is continuous on $[0,\infty)$ and differentiable on
  $(0,+\infty)$, and for $\lambda>0$,
  \[
    \varphi'(\lambda)
    =
    \begin{cases}
      \ds\frac{s\tau_{d-1}}{R^s}\int_{-1}^1
      \frac{(t-\lambda)(1-t^2)^{\frac{d-3}{2}}}{(\lambda^2+1-2\lambda t)^{\frac{s+2}{2}}}\dd t
      +\gamma\alpha R^\alpha\lambda^{\alpha-1}
      & \text{if $s\neq0$}\\[1em]
      \ds\tau_{d-1}
      \int_{-1}^1\frac{(t-\lambda)(1-t^2)^{\frac{d-3}{2}}}{\lambda^2+1-2\lambda t}\dd t
      +\gamma\alpha R^\alpha\lambda^{\alpha-1}
      & \text{if $s=0$}
    \end{cases}.
  \]
\end{lemma}

\begin{proof}
  The result follows from the identity
  \[
    \int_{\sph_R} K_s(x-y) \sigma_R(\dd y)
    =
    \begin{cases}
      \ds R^{-s}\int_{\sph}(\lambda^2-2\lambda \xh\cdot \yh+1)^{-\frac{s}{2}}\sigma_1(\dd \yh)
      &\text{if $s\neq0$}\\[1em]
      \ds-\log R-\frac{1}{2}\int_{\sph} \log(\lambda^2-2\lambda \xh\cdot \yh+1) \sigma_1(\dd \yh)
      &\text{if $s=0$}
    \end{cases} ,
  \]
  which holds for arbitrary $\xh\in\sph$, and the Funk\,--\,Hecke formula of
  Lemma \ref{le:funkhecke}.
\end{proof}

\begin{lemma}
  \label{le:spheres:opt}
  Assume that
  \[
    d\geq3,\quad
    -2 < s \leq d-4,\quad
    V = \gamma \ABS{\cdot}^\alpha,\ \alpha\geq2.
  \]
  The energy function $\eta(R):= \mathrm{I}(\sigma_R)$, from
  \eqref{eq:Esalpha}, achieves its infimum on $(0,\infty)$ at the unique point
  \[
    R:=
    \begin{cases}
      \PAR{\frac{sW_{s,d-1}}{2\gamma\alpha}}^{\frac{1}{s+\alpha}}
      &\text{if $s\neq0$}\\
      \PAR{\frac{1}{2\gamma\alpha}}^{\frac{1}{\alpha}}
      &\text{if $s=0$}
    \end{cases},
  \]
  where $W_{s,d-1}$ is the Wiener constant for $\sph$ given by
  \[
    W_{s,d-1}
    :=\iint K_s(x-y) \sigma_1(\dd x) \sigma_1(\dd y)
    =
    \begin{cases}
      \ds\mathrm{sign}(s)\frac{\Gamma(\frac{d}{2})\Gamma(d-1-s)}{\Gamma(\frac{d-s}{2})\Gamma(d-1-\frac{s}{2})}&
      \text{if $-2<s<0$, or}\\[-1ex]
      & \text{if $0<s<d-1$ and $d\geq 4$}\\[1em]
      \ds-\log(2)+\frac{\psi(d-1)-\psi(\frac{d-1}{2})}{2}&\text{if $s=0$ and
        $d\geq 4$}
    \end{cases},
  \]
  and $\psi(z):=\Gamma'(z)/\Gamma(z)$ is the digamma special function, see for
  example \cite[Prop.\ 4.6.4, p.\ 180]{MR3970999}.
\end{lemma}

Note that $-2<s\leq d-4$ implies $d>2$; hence $d\geq3$.

We also remark that when $s=0$, the radius $R$ does not depend on the dimension $d$.

\begin{proof}
  We have
  \begin{align*}
    \eta(R)
    : =\mathrm{I}(\mu_R)
    &= \iint_{\sph_R\times\sph_R} K_s(x-y)\sigma_R(\dd x)\sigma_R(\dd y)
      + 2\int_{\sph_R} V(x)\sigma_R(\dd x)\\
    &=\iint_{\sph\times\sph} K_s(R\xh-R\yh)\sigma_1(\dd \xh)\sigma_1(\dd \yh) + 2 \gamma R^\alpha.
  \end{align*}

  \emph{Case $s=0$.} In this case, we have
  $K_s(x - y) = K_0(R(\xh-\yh)) = -\log(R)-\log|\xh-\yh|$,
  and thus
  \[
    \eta(R) = -\log(R) - \iint_{\sph\times\sph}
     \log|\xh-\yh|\sigma_1(\dd \xh) \sigma_1(\dd \yh) + 2\gamma R^\alpha.
  \]
  Now $\eta$ is strictly convex and reaches its minimum at the unique optimal
  point
  \begin{equation}\label{eq:R0}
    R_*=\PAR{\frac{1}{2 \gamma \alpha}}^{1/\alpha}.
  \end{equation}

  \emph{Case $-2<s<d-4$, $s\neq0$.} 
  \[
    \eta(R) =
    \frac{W_{s,d-1}}{R^s}
     + 2 \gamma R^\alpha
     \quad\mbox{where}\quad
     W_{s,d-1} = \iint_{\sph\times\sph}K_s(\xh-\yh)
     \sigma_1(\dd \xh)\sigma_1(\dd \yh) .
  \]
  The equation $\eta'(R)=0$ has a unique solution (critical point) given by
  \begin{equation}\label{eq:Rs}
    R_*=\PAR{\frac{sW_{s,d-1}}{2 \gamma \alpha}}^{1/(s+\alpha)}.
  \end{equation}
  We have
  \begin{align*}
    \eta''(R)
    &= s(s+1) R^{-(s+2)}W_{s,d-1} + 2\gamma \alpha(\alpha-1)R^{\alpha-2}\\
    &=R^{-(s+2)}\PAR{s(s+1) W_{s,d-1} + 2\gamma\alpha(\alpha-1)R^{\alpha+s}},
  \end{align*}
  and thus
  \[
    \eta''(R_*)
    =\PAR{\frac{2 \gamma \alpha}{sW_{s,d-1}}}^{\frac{s+2}{\alpha+s}}
    s(s + \alpha) W_{s,d-1}. 
  \]
  It follows that $\eta''(R_*)>0$ since $s+\alpha>0$ and since $s$ and
  $W_{s,d-1}$ have the same sign.
\end{proof}

\subsection{Case $\alpha\geq2$}\label{se:alpha>=2}

Let $V = \gamma\left|\cdot\right|^\alpha$ with $\gamma>0$ and $\alpha\geq2$.
We have to show that $\meq$ is uniform on a sphere. For this purpose, we
verify the Frostman conditions \eqref{eq:eulerlagrange} which asserts that the
support of $\meq$ is a sphere of radius $R$ if and only if, for some constant
$c$, we have $\varphi(\lambda)=c$ when $\lambda=1$ and
$\varphi(\lambda)> c$ when $\lambda\neq1$. Since $\varphi$ is continuous on
$[0,+\infty)$ and differentiable on $(0,+\infty)$, part \ref{it:thma}-\ref{it:sphere} of Theorem~\ref{th:main} follows from
Lemma~\ref{le:phi':d>=4s=0a>=2} in the case where $s=0$ and $d\geq4$, part
\ref{it:thmb}-\ref{it:sph3} from Lemma~\ref{le:phi':d=3s=-1a>=2} in the case
where $s=-1$ and $d=3$, and Lemma~\ref{le:phi':d>=5s=d-4a>=2} in the case
$d\geq 5$ and $s=d-4$.

\begin{lemma}\label{le:phi':d>=4s=0a>=2}
  Let $\varphi$ be as in \eqref{eq:sphere:phi} with $s=0$, $d\geq4$,
  $\alpha\geq2$, and $R=(\frac{1}{2\gamma\alpha})^{\frac{1}{\alpha}}$. Then
  \[
    \varphi'(\lambda)
    =\frac{1}{2\lambda}\SBRA{\frac{1-\lambda^2}{1+\lambda^2}%
      \ \hg\PAR{\frac{1}{2},1;\frac{d}{2};\frac{4\lambda^2}{(1+\lambda^2)^2}}-1}+\frac{\lambda^{\alpha-1}}{2},\quad\lambda>0.
  \]
  Moreover $\varphi'(\lambda)<0$ if $0<\lambda<1$, $\varphi'(1)=0$, while $\varphi'(\lambda)>0$ if $\lambda>1$.
\end{lemma}

\begin{proof}
  In view of the formula given by Lemma \ref{le:spheres}, we have, with $\zeta:=4\lambda^2/(1+\lambda^2)^2$,
  \begin{align*}
    J
    &:=\int_{-1}^1\frac{(t-\lambda)(1-t^2)^{\frac{d-3}{2}}}{\lambda^2+1-2\lambda
      t}\dd t\\
    &=
      \frac{\zeta}{4\lambda}\int_{0}^1\frac{2t^2-(1+\lambda^2)}{1-\zeta t^2}(1-t^2)^{\frac{d-3}{2}}\dd
      t\\
    &=
      \frac{\zeta}{4\lambda}\int_{0}^1\frac{2u^{1/2}-u^{-1/2}(1+\lambda^2)}{1-\zeta u}(1-u)^{\frac{d-3}{2}}\dd
      u.
  \end{align*}
  The Euler integral representation of Lemma \ref{le:hypeuler} gives in
  particular, for $a>-1$, $b>-1$, $|\zeta|<1$,
  \[
    \int_0^1\frac{u^{a-1}(1-u)^{b-1}}{1-\zeta u}\dd u
    =\frac{\Gamma(a)\Gamma(b)}{\Gamma(a+b)}\ \hg(a,1;a+b;\zeta).
  \]
  Using this formula with $a\in\{3/2,1/2\}$, $b=(d-1)/2$, $c=1$, and
  $\Gamma(\zeta+1)=\zeta\Gamma(\zeta)$, we get
  \[
    \tau_{d-1}J =\frac{\zeta}{4\lambda} \SBRA{
      \frac{2}{d}\ \hg\PAR{\frac{3}{2},1;\frac{d+2}{2};\zeta} -(1+\lambda^2)\
      \hg\PAR{\frac{1}{2},1;\frac{d}{2};\zeta}}.
  \]
  At this step, we observe that the identity $\zeta(\zeta+1)_k=(\zeta)_{k+1}$ gives the
  formula
  \[
    1+\frac{a}{b}\zeta\ \hg(a+1,1;b+1;\zeta) = \hg(a,1;b;\zeta).
  \]
  Finally, using this formula, we obtain, denoting
  $G:= \hg\PAR{\frac{1}{2},1;\frac{d}{2},\zeta} := \sum_{k=0}^\infty\frac{(\frac{1}{2})_k}{(\frac{d}{2})_k}\zeta^k$,
  \[
    \tau_{d-1}J
    =\frac{\zeta}{4\lambda}
    \SBRA{\frac{2}{\zeta}(G-1)-(1+\lambda^2)G}
    =\frac{1}{2\lambda}
    \SBRA{\frac{1-\lambda^2}{1+\lambda^2}G-1}.
  \]
  Hence the formula for $\varphi'(\lambda)$. Alternatively, we could also use the formula for
  $\varphi$ of Lemma \ref{le:spheres} to get
  \begin{equation}
    \varphi(\lambda)
    =\frac{\log(2\gamma\alpha)}{\alpha}%
    +\frac{\lambda^2}{d(1+\lambda^2)^2}%
    \
    \hgp{3}{2}\PAR{1,1,\frac{3}{2};2,\frac{d+2}{2};\frac{4\lambda^2}{(1+\lambda^2)^2}}%
    -\frac{\log(1+\lambda^2)}{2}
    +\frac{\lambda^\alpha}{2\alpha}.
  \end{equation}

At this step, we observe that the formula in the statement of the lemma for $\varphi'$ gives $\varphi'(1)=0$
as the parameters of the $\hgp{3}{2}$ function ensure that it is finite (see Lemma~\ref{le:EulerIntHG} or \cite[eqn. 15.4.20]{NIST:DLMF2022}).
It remains to determine the sign of $\varphi'(\lambda)$ for
  $\lambda\neq1$. Let us consider first the case $d=4$. As 
  \begin{equation}
    \hg\PAR{\frac{1}{2},1;2;\zeta}
    =\sum_{k=0}^\infty\frac{(\frac{1}{2})_k}{(k+1)!}\zeta^k
    =2\frac{1-\sqrt{1-\zeta}}{\zeta},
  \end{equation}
  we get 
  \[
    \hg\PAR{\frac{1}{2},1;2;\zeta}
    =\begin{cases}
      \ds 1+\lambda^2 & \text{if $0\leq\lambda\leq1$}\\[.25em]
      \ds 1+\frac{1}{\lambda^2} & \text{if $\lambda\geq1$}
    \end{cases},
  \]
  and therefore
  \[
    \varphi'(\lambda)
    =\begin{cases}
      \ds -\frac{\lambda}{2}+\frac{\lambda^{\alpha-1}}{2} & \text{if $0\leq \lambda\leq1$}\\[.5em]
      \ds \frac{1-2\lambda^2}{2\lambda^3}+\frac{\lambda^{\alpha-1}}{2} & \text{if $\lambda\geq1$}
    \end{cases}.
  \]
  Alternatively, we could use Lemma \ref{le:phiop} which replaces the Taylor
  series expansion behind the hypergeometric based formulas by the generating
  series of orthogonal polynomials. From this it can be checked that
  $\varphi'(\lambda)<0$ if $0<\lambda<1$ while $\varphi'(\lambda)>0$ if
  $\lambda>1$, first when $\alpha=2$ and then by monotony for all
  $\alpha\geq2$. This proves the desired result for $d=4$.
  For the general
  case $d\geq 4$, noting that $\hg\PAR{\frac{1}{2},1;\frac{d}{2}; z}$, $z\in[0, 1]$, decreases as $d$ increases, an examination of the formula for $\varphi'(\lambda)$
  reveals that as $d$ increases, then $\varphi'(\lambda)$ decreases when $0<\lambda<1$,
  while $\varphi'(\lambda)$ increases when $\lambda>1$,
  which reduces the analysis to the case $d=4$.
\end{proof}

\begin{lemma}\label{le:phi':d=3s=-1a>=2}
  Let $\varphi$ be as in \eqref{eq:sphere:phi}, $(d,s)=(3,-1)$, $\alpha\geq2$,
  and $R=(\frac{2}{3\gamma\alpha})^{\frac{1}{\alpha-1}}$. Then
  \begin{align*}
    \varphi'(\lambda)
    &=R\begin{cases}
      -\frac{2}{3}\lambda+\frac{2}{3}\lambda^{\alpha-1}
      &\text{if $0<\lambda\leq1$}\\
      -1+\frac{1}{3\lambda^2}+\frac{2}{3}\lambda^{\alpha-1}
      &\text{if $\lambda>1$}
    \end{cases}.
  \end{align*}
  Moreover $\varphi'(\lambda)\leq0$ if $0<\lambda<1$, $\varphi'(1)=0$, while
  $\varphi'(\lambda)>0$ if $\lambda>1$.
\end{lemma}

\begin{proof}
  From Lemma \ref{le:spheres} we get, for $\lambda\geq0$,
  \begin{align*}
    \varphi(\lambda)
    &=-\frac{R}{2}\int_{-1}^1(\lambda^2-2\lambda t+1)^{1/2}\dd t+\gamma(\lambda R)^\alpha\\
    &=-R\frac{|\lambda-1|^3-(\lambda+1)^3}{6\lambda}+\gamma(\lambda R)^\alpha\\
    &=R\begin{cases}
      -(1+\frac{\lambda^2}{3})+\gamma\lambda^\alpha R^{\alpha-1}&\text{if $0\leq\lambda\leq1$}\\
      -(\lambda+\frac{1}{3\lambda})+\gamma\lambda^\alpha R^{\alpha-1}&\text{if $\lambda>1$}
    \end{cases}.
  \end{align*}
  Hence
  \[
    \varphi'(\lambda)
    =R\begin{cases}
      -\frac{2}{3}\lambda+\gamma\alpha(\lambda R)^{\alpha-1}
      &\text{if $0<\lambda\leq1$}\\
      -1+\frac{1}{3\lambda^2}+\gamma\alpha(\lambda R)^{\alpha-1}
      &\text{if $\lambda>1$}
    \end{cases}.
  \]
  From now on we take $R=(\frac{2}{3\gamma\alpha})^{\frac{1}{\alpha-1}}$,
  which makes the critical value of $\lambda$ on $(0,1]$ equal to $1$. Hence
  \[
    \varphi'(\lambda)
    =R\begin{cases}
      -\frac{2}{3}\lambda+\frac{2}{3}\lambda^{\alpha-1}
      &\text{if $0<\lambda\leq1$}\\
      -1+\frac{1}{3\lambda^2}+\frac{2}{3}\lambda^{\alpha-1}
      &\text{if $\lambda>1$}
    \end{cases}.
  \]
  Now $\varphi'(1)=0$, and moreover $\varphi'(\lambda)<0$ for $0<\lambda<1$
  when $\alpha>2$ while $\varphi'(\lambda)=0$ for $0<\lambda<1$ when
  $\alpha=2$. If $\lambda>1$, then
  $(\lambda^2\varphi'(\lambda))'=R(-2\lambda+\frac{2}{3}(\alpha+1)\lambda^\alpha)>0$
  when $\lambda>(3/(1+\alpha))^{1/(\alpha-1)}$, and this last value is $\leq1$
  since $\alpha\geq2$, which implies that $\varphi'(\lambda)>0$ for
  $\lambda>1$.
\end{proof}

\begin{remark}[General $s$]\label{rm:phiphis}
  Suppose that $s>-2$, $s\neq0$, and $d\geq4$. By proceeding as in the proof
  of Lemma \ref{le:phi':d>=4s=0a>=2}, it is possible to obtain the following
  formulas, for all $\lambda\geq0$,
  \[
    \tau_{d-1}%
    \int_{-1}^1\frac{(1-t^2)^{\frac{d-3}{2}}}{(\lambda^2+1-2\lambda t)^{\frac{s}{2}}}%
    \dd t
    =
      (1+\lambda^2)^{-\frac{s}{2}}%
      \
      \hg\PAR{\frac{s}{4},\frac{s+2}{4};\frac{d}{2};\frac{4\lambda^2}{(1+\lambda^2)^2}}
    \]
    and
    \begin{multline*}
      \tau_{d-1}%
      \int_{-1}^1\frac{(t-\lambda)(1-t^2)^{\frac{d-3}{2}}}{(\lambda^2+1-2\lambda t)^{\frac{s+2}{2}}}%
      \dd t\\=
      \frac{\lambda}{d(1+\lambda^2)^{\frac{s+4}{2}}}\Bigg[-\frac{(1+\lambda^2)}{d}\
      \hg\PAR{\frac{s+2}{4},\frac{s+4}{4};\frac{d}{2};\frac{4\lambda^2}{(1+\lambda^2)^2}}\\%
      +(2+s)\ \hg\PAR{\frac{s+2}{4}+1,\frac{s+4}{4};\frac{d}{2}+1;\frac{4\lambda^2}{(1+\lambda^2)^2}}\Bigg].
  \end{multline*}
  This gives formulas for $\varphi(\lambda)$ and $\varphi'(\lambda)$ via Lemma
  \ref{le:spheres}. Unfortunately, the formula for $\varphi'(\lambda)$ does
  not seem to be monotonic with respect to $d$, and thus one cannot proceed as
  in proof of Lemma \ref{le:phi':d>=4s=0a>=2}.
\end{remark}

When $s$ is an integer, instead of using power series and hypergeometric
functions for the evaluation of the integrals in the formulas for $\varphi$
and $\varphi'$ of Lemma \ref{le:spheres}, we could use alternatively
orthogonal polynomials, which leads for instance when $s=0$ to the
trigonometric formulas of Lemma \ref{le:phiop}.

\begin{lemma}[Trigonometric formulas]\label{le:phiop}
  Let $\varphi$ be as in \eqref{eq:sphere:phi} with $s=0$, $d=4+2m$ where $m$
  is a non-negative integer, and
  $R=(\frac{1}{2\gamma\alpha})^{\frac{1}{\alpha}}$. Then
  \[
    \varphi'(\lambda)
    =\begin{cases}
      \ds\frac{\tau_{2m+3}}{2}\sum_{n=0}^{m}\lambda^{2n+1}\int_0^\pi\sin(\theta)^{2m+1}(\sin((2n+3)\theta)-\sin((2n+1)\theta))\dd\theta %
      +\frac{\lambda^{\alpha-1}}{2}
      &\text{if $0\leq\lambda\leq1$}\\[1em]
      \ds\frac{\tau_{2m+3}}{2}\sum_{n=0}^{m}\frac{1}{\lambda^{2n+3}}\int_0^\pi
      \sin(\theta)^{2m+1}(\sin((2n+3)\theta)-\sin((2n+1)\theta)(2\lambda^2-1))\dd\theta
      +\frac{\lambda^{\alpha-1}}{2}
      &\text{if $\lambda\geq1$}
    \end{cases},
  \]
  where $\tau$ is as in Lemma \ref{le:funkhecke}. In particular, when $d=4$
  ($m=0$), we find
  \[
    \varphi'(\lambda)
    =\begin{cases}
      \ds-\frac{\lambda}{2}+\frac{\lambda^{\alpha-1}}{2}
      &\text{if $0\leq\lambda\leq1$}\\[.5em]
      \ds\frac{1-2\lambda^2}{2\lambda^3}+\frac{\lambda^{\alpha-1}}{2}
      &\text{if $\lambda\geq1$}
    \end{cases},
  \]
  while when $d=6$ ($m=1$), we find
  \[
    \varphi'(\lambda)
    =\begin{cases}
      \ds-\frac{2\lambda}{3}+\frac{\lambda^3}{6}+\frac{\lambda^{\alpha-1}}{2}
      &\text{if $0\leq\lambda\leq1$}\\[.5em]
      \ds\frac{-1+4\lambda^2-6\lambda^4}{6\lambda^5}+\frac{\lambda^{\alpha-1}}{2}
      &\text{if $\lambda\geq1$}
    \end{cases}.
  \]
\end{lemma}

\begin{proof}
  In view of Lemma \ref{le:spheres}, it suffices to compute
  \[
    I(\lambda):=\int_{-1}^1\frac{t-\lambda}{\lambda^2+1-2\lambda
      t}(1-t^2)^{\frac{d-3}{2}}\dd t.
  \]
  Let ${(U_n)}_{n\geq0}$ be the Chebyshev orthogonal polynomials of the second
  kind\footnote{Three terms recurrence relation
    $U_{n+1}(t)=2tU_n(t)-U_{n-1}(t)$, $n\geq1$, with $U_0(t)=1$ and
    $U_1(t)=2t$.}, orthogonal with respect to the semicircle weight
  $t\mapsto\sqrt{1-t^2}$ on $[-1,1]$. In order to compute $I(\lambda)$, the
  idea is to exploit their generating series formula, which states, for
  $|t|<1$ and $|\lambda|<1$,
  \[
    \frac{1}{1+\lambda^2-2\lambda t}
    =\sum_{n=0}^\infty U_n(t)\lambda^n.
  \]
  The orthogonality relation states, for all
  polynomial $P$ of degree $k\geq0$ and all $n>k$,
  \[
    \int_{-1}^1P(t)U_n(t)\sqrt{1-t^2}\,\dd t=0.
  \]
  Now, since $m:=\frac{d-4}{2}$ is a non-negative integer, the expression
  $(t-\lambda)(1-t^2)^m$ is a polynomial of degree $k=2m+1$ with respect to
  $t$, and therefore, when $|\lambda|<1$,
  \[
    I(\lambda)
    =\sum_{n=0}^{2m+1}\lambda^n%
    \int_{-1}^1(t-\lambda)(1-t^2)^{m}U_n(t)\sqrt{1-t^2}\,\dd t,
  \]
  where we have used crucially the identity
  $(1-t^2)^{\frac{d-3}{2}}=(1-t^2)^m\sqrt{1-t^2}$. Now, to evaluate the
  integral in the right-hand side above, we use the trigonometric change of
  variable $t=\cos(\theta)$ and the fact that
  $U_n(\cos(\theta))=\sin((n+1)\theta)/\sin(\theta)$, which give
  \[
    \int_{-1}^1(t-\lambda)(1-t^2)^{m}U_n(t)\sqrt{1-t^2}\,\dd t
    =
    \int_0^\pi(\cos(\theta)-\lambda)\sin(\theta)^{2m+1}\sin((n+1)\theta)\dd\theta.
  \]
  It follows that if $0<\lambda<1$, then using standard trigonometric formulas,
  \[
    I(\lambda)=\frac{1}{2}\sum_{n=0}^{m}\lambda^{2n+1}\int_0^\pi
      \sin(\theta)^{2m+1}(\sin((2n+3)\theta)-\sin((2n+1)\theta))\dd\theta.
  \]
  This produces the desired formula for $\varphi'(\lambda)$ when $0<\lambda<1$.

  Let us establish now the formula when $\lambda>1$. Let us set
  $\rho:=1/\lambda$. Then we have
  \[
    I(\lambda)
    =
    \rho^2I(\rho)
    +\rho(\rho^2-1)\int_{-1}^1\frac{(1-t^2)^{\frac{d-3}{2}}}{1+\rho^2-2\rho t}\dd t.
  \]
  Since $0<\rho<1$, proceeding as before for $I(\rho)$, we get for the last integral
  \begin{align*}
    \int_{-1}^1\frac{(1-t^2)^{\frac{d-3}{2}}}{1+\rho^2-2\rho t}\dd t
    &=\sum_{n=0}^{m}\rho^{2n}\int_0^\pi
      \sin(\theta)^{2m+1}\sin((2n+1)\theta)\dd\theta
  \end{align*}
  where the last step comes from symmetry of $\sin$. Hence the desired formulas
  for $\varphi'(\lambda)$.
\end{proof}

\begin{lemma}\label{le:phi':d>=5s=d-4a>=2}
  Let $\varphi$ be as in \eqref{eq:sphere:phi} with $d\geq5$, $s=d-4$,
  $\alpha\geq2$, and $R$ as in Theorem \ref{th:main}. Then
  \[
    \varphi'(\lambda)
    =\begin{cases}
      \ds
      \frac{2s}{(s+4)R^s}(-\lambda+\lambda^{\alpha-1})
      &\text{if $0\leq\lambda\leq1$}\\[1em]
      \ds
      \frac{s}{(s+4)R^s}\Bigr(\frac{2-(s+4)+(s+4)\lambda^2+2\lambda^{s+\alpha+2}}{\lambda^{s+3}}\Bigr)
      &\text{if $\lambda\geq1$}
    \end{cases}.
  \]
  Moreover $\varphi'(\lambda)<0$ if $0<\lambda<1$, $\varphi'(1)=0$, while $\varphi'(\lambda)>0$ if $\lambda>1$.
\end{lemma}

\begin{proof}
  From Lemma \ref{le:spheres} we get
  \[
    \varphi'(\lambda)
    = -\frac{s\tau_{s+3}}{R^s}\int_{-1}^1\frac{(\lambda-t)(1-t^2)^{\frac{s+1}{2}}}{(\lambda^2+1-2\lambda
        t)^{\frac{s+2}{2}}}\dd t+\gamma\alpha R^\alpha\lambda^{\alpha-1}.
  \]
  The idea is to imitate the proof of Lemma \ref{le:phiop}, and compute
  $\varphi'(\lambda)$ using the following generating series, valid for all
  $|t|<1$ and $|\lambda|<1$,
  \[
    \frac{1}{(1+\lambda^2-2\lambda t)^{\ell}}
    =\sum_{n=0}^\infty C^{(\ell)}_n(t)\lambda^n
  \]
  where $(C^{(\ell)}_n)_{n\geq0}$ are the Gegenbauer ultraspherical
  polynomials\footnote{Recurrence relation
    $C^{(\ell)}_{n}(t)=\frac{2t(n+\ell-1)}{n}C^{(\ell)}_{n-1}(t)-(n+2\ell-2)C^{(\ell)}_{n-2}(t)$,
    $n\geq2$, $C^{(\ell)}_0=1$, $C^{(\ell)}_1(t)=2\ell x$. Include Chebyshev
    (both kinds) and Legendre polynomials as special cases with
    $\ell\in\{0,1,1/2\}$.} of parameter $\ell:=\frac{s+2}{2}$, orthogonal on
  $[-1, 1]$ with respect to the measure
  $\dd\mu(t) = (1-t^2)^{\ell-\frac{1}{2}} \dd t
  =(1-t^2)^{\frac{s+1}{2}} \dd t$.
  The choice of $\ell$ is dictated by the formula for $\varphi'(\lambda)$
  above.
  Using $C_0^{(\ell)}=1$, $C_1^{(\ell)}(t)=(s+2)t$ and orthogonality gives
  \begin{align*}
    \int_{-1}^1\frac{(\lambda-t)(1-t^2)^{\frac{s+1}{2}}}{(\lambda^2+1-2\lambda
    t)^{\frac{s+2}{2}}}\dd t
    &=\sum_{n=0}^\infty\lambda^n\int_{-1}^1(\lambda-t)C_n^{(\ell)}(t)\dd\mu(t)\\
    &=\lambda\int_{-1}^1(1-t^2)^{\frac{s+1}{2}}\dd t
      -\lambda\int_{-1}^1tC_1^{(\ell)}(t)\dd\mu(t)\\
    &=\frac{\lambda}{\tau_{s+3}}-\lambda(s+2)\int_{-1}^1t^2(1-t^2)^{\frac{s+1}{2}}\dd t\\
    &=\frac{\lambda}{\tau_{s+3}}-\lambda(s+2)\frac{\Gamma(3/2)\Gamma((s+3)/2)}{\Gamma(3+s/2)}\\
    &=\frac{2\lambda}{(s+4)\tau_{s+3}}.
  \end{align*}
  Hence, for $0<\lambda<1$, we obtain
  \begin{equation*}
    \varphi'(\lambda)
    =-\frac{2s}{(s+4)R^s}\lambda+\gamma\alpha R^\alpha\lambda^{\alpha-1}
    =\frac{2s}{(s+4)R^s}\Bigr(-\lambda+\gamma\alpha\frac{s+4}{2s}R^{s+\alpha}\lambda^{\alpha-1}\Bigr)
  \end{equation*}
  which leads to the desired formula when we take $R=(\frac{2s}{(s+4)\gamma\alpha})^{\frac{1}{s+\alpha}}$.

  Let us consider now the case $\lambda>1$. Denoting
  $\rho:=1/\lambda$, we have 
  \begin{multline*}
    \int_{-1}^1\frac{(\lambda-t)(1-t^2)^{\frac{s+1}{2}}}{(\lambda^2+1-2\lambda
      t)^{\frac{s+2}{2}}}\dd t
    =
    \rho^{s+2}\int_{-1}^1
    \frac{(\rho-t)(1-t^2)^{\frac{s+1}{2}}}
    {(1+\rho^2-2\rho t)^{\frac{s+2}{2}}}\dd t
    +
    \rho^{s+1}(1-\rho^2)\int_{-1}^1
    \frac{(1-t^2)^{\frac{s+1}{2}}}
    {(1+\rho^2-2\rho t)^{\frac{s+2}{2}}}\dd t.
  \end{multline*}
  Now, using the fact that $0<\rho<1$, we get, from the previous computations,
  \[
    \int_{-1}^1\frac{(1-t^2)^{\frac{s+1}{2}}}{(1+\rho^2-2\rho
      t)^{\frac{s+2}{2}}}\dd t
    =\int_{-1}^1(1-t^2)^{\frac{s+1}{2}}\dd t
    =\int_0^1u^{-1/2}(1-u)^{\frac{s+1}{2}}\dd u
    =\frac{1}{\tau_{s+3}}
  \]
  and
  \begin{align*}
    \varphi'(\lambda)
    &=-\frac{s}{R^s}\Bigr(\frac{2}{(s+4)\lambda^{s+3}}-\frac{1}{\lambda^{s+3}}+\frac{1}{\lambda^{s+1}}\Bigr)
      +\gamma\alpha R^\alpha\lambda^{\alpha-1}\\
    &=
      \frac{2s}{(s+4)R^s}
      \Bigr(\frac{2-(s+4)+(s+4)\lambda^2}{2\lambda^{s+3}}
      +\gamma\alpha\frac{s+4}{2s}R^{s+\alpha}\lambda^{\alpha-1}\Bigr).
  \end{align*}
  Hence with $R=(\frac{2s}{(s+4)\gamma\alpha})^{\frac{1}{s+\alpha}}$ we find,
  for $\lambda>1$,
  \[
    \varphi'(\lambda)
    =\frac{s}{(s+4)R^s}\Bigr(\frac{2-(s+4)+(s+4)\lambda^2+2\lambda^{s+\alpha+2}}{\lambda^{s+3}}\Bigr).
  \]

  The method works more generally when $m:=(d-4-s)/2$ is a non-negative
  integer, by using
  $(1-t^2)^{\frac{d-3}{2}}=(1-t^2)^{\frac{d-4-s}{2}}(1-t^2)^{\frac{s+1}{2}}$ where
  $(1-t^2)^{\frac{d-4-s}{2}}=(1-t^2)^m$ is then a polynomial in $t$.

  Note that
  $\lim_{\lambda\to1^+}\varphi'(\lambda)=\lim_{\lambda\to1^-}\varphi'(\lambda)=0$.
  If $0<\lambda<1$ then $\varphi'(\lambda)<0$, while if $\lambda>1$, then the
  derivative of the numerator of the fraction in the formula for
  $\varphi'(\lambda)$ is
  \begin{align*}
    2(\alpha+s+2)\lambda^{\alpha+s+1}-2(s+4)\lambda
    &=2\lambda((\alpha+s+2)\lambda^{\alpha+s}-(s+4))\\
    &>2((\alpha+s+2)-(s+4))\geq0,
  \end{align*}
  hence $\varphi'(\lambda)>0$, which completes the proof.
\end{proof}

\subsection{Case $0 < \alpha < 2$}\label{se:alpha02}

Let $d \geq 4$ and $s = d - 4$. For an arbitrary $R>0$ define
\[
  \mu:=\beta fm_d+(1-\beta)\sigma_R,
\]
where $\beta:=\frac{2-\alpha}{s+2}$ and
$f(x):=\frac{\alpha+s}{R^{\alpha+s}|\sph|}|x|^{\alpha-4}\mathbf{1}_{|x|\leq
  R}$.
The condition $0<\alpha<2$ ensures that $0<\beta<1$ so $\mu$ is a probability
measure. Since $K_s*\mu+V$ is radially symmetric, for $x\in\mathbb{R}^d$,
$x=\lambda R \xh$ with $\lambda>0$ and $\xh\in\sph$, the modified potential
is 
\begin{equation}\label{eq:d>=4,s=d-4,0<a<2}
  \varphi(\lambda)
  =\varphi_\mu(\lambda)
  :=(K_s*\mu+V)(\lambda R \xh)
  = \int  K_s(\lambda R \xh, y) \mu(\dd y) 
    +\gamma R^\alpha\lambda^\alpha.
\end{equation}
The Frostman conditions are satisfied if we show that for some constant $c$,
we have $\varphi(\lambda)=c$ if $0\leq\lambda\leq1$ while
$\varphi(\lambda)\geq c$ if $\lambda\geq1$. Since $\varphi$ is continuous on
$[0,+\infty)$, and differentiable on $(0,+\infty)$, the desired result follows
from the next Lemma. 

\begin{lemma}\label{le:d>=4,s=d-4,0<a<2}
  Let $\varphi$ be as in \eqref{eq:d>=4,s=d-4,0<a<2} with $d\geq 5$, $s=d-4$,
  $0<\alpha<2$, and $R$ as in Theorem \ref{th:main}. Then
  \[
    \varphi'(\lambda)
    =\begin{cases}
      0
      & \text{if $0<\lambda<1$}\\
      \geq0
      & \text{if $\lambda>1$}
    \end{cases}.
  \]
\end{lemma}

\begin{proof}
  Let us focus first on the case $d=4$ ($s=0$). We have
  \[
    - \int\log\PAR{|x-y|^2}\mu(\dd y) =I_1(\lambda)+I_2(\lambda)
  \]
  where, using $y = r R \yh$ for an arbitrary $\yh\in\dS^3$,
  \begin{align*}
    I_1(\lambda)
    &:=- \alpha 
    \PAR{1-\frac{\alpha}{2}}
      \int_{\dS^3}\int_0^1r^{\alpha-1}\log\PAR{\lambda^2R^2+r^2R^2-2\lambda R^2r\DOT{z}{u}}
      \dd r\sigma_1(\dd u)\\
    I_2(\lambda)&:=- \frac{\alpha}{2} 
                  \int_{\dS^3}\log\PAR{\lambda^2R^2+R^2-2\lambda R^2\DOT{z}{u}}
                  \sigma_1(\dd u).
  \end{align*}
  By the Funk\,--\,Hecke formula (Lemma \ref{le:funkhecke}), we have
  \begin{align*}
    I_2(\lambda)
    &=-\frac{\alpha}{4}\tau_3\int_{-1}^1\log\PAR{\lambda^2R^2+R^2-2\lambda R^2t}\sqrt{1-t^2}\,\dd t,\\
    I_2'(\lambda)
    &=\frac{\alpha}{2}\tau_3\int_{-1}^1\frac{t-\lambda}{\lambda^2+1-2\lambda t}\sqrt{1-t^2}\,\dd t,
  \end{align*}
  where $\tau_3=\frac{2}{\pi}$ is as in Lemma \ref{le:funkhecke}. Now we
  consider two cases, $0<\lambda<1$, and $\lambda>1$. If $0<\lambda<1$, then
  using the generating
  function\footnote{$(1-2tu+u^2)^{-1}=\sum_{n=0}^\infty U_n(t)u^n$ for all
    $|u|<1$, $U_{n+1}(t)=2tU_n(t)-U_{n-1}(t)$, $U_0:=1$, $U_1(t):=2t$.} for
  the second kind Chebyshev polynomials ${(U_n)}_{n\geq0}$, we get
  \[
    I_2'(\lambda)
    =\frac{\alpha}{2}\tau_3
    \int_{-1}^1\PAR{t-\lambda}\sum_{n=0}^\infty
    U_n(t)\lambda^n\sqrt{1-t^2}\,\dd t
    =-\frac{\alpha}{4}\lambda
  \]
  using $U_0=1$, $U_1(t)=2t$, and the orthonormality relation
  $\tau_3\int_{-1}^1U_n(t)U_m(t)\sqrt{1-t^2}\,\dd t=\ind_{n=m}$.

  Next, if $\lambda>1$, then $0<1/\lambda<1$ and by using the same method we get
  \begin{align*}
    I_2'(\lambda)
    &=\frac{\alpha}{2}\tau_3\frac{\lambda}{\lambda^2}
      \int_{-1}^1\frac{\frac{1}{\lambda}t-1}{1+\PAR{\frac{1}{\lambda}}^2-2\frac{1}{\lambda}t}
      \sqrt{1-t^2}\,\dd t\\
    &=\frac{\alpha}{2\lambda}\tau_3
      \int_{-1}^1\PAR{\frac{1}{\lambda}t-1}\sum_{n=0}^\infty
      U_n(t)\PAR{\frac{1}{\lambda}}^n\sqrt{1-t^2}\,\dd t\\
    &=\frac{\alpha}{2\lambda}\Bigr(\frac{1}{2\lambda^2}-1\Bigr).
  \end{align*}

  Let us consider now $I_1(\lambda)$. We have
  $I_1(\lambda)=\frac{\alpha}{2}(1-\frac{\alpha}{2})J(\lambda)$ where
  \[
    J(\lambda)
    :=-\int_0^1\PAR{\int_{\dS^3}\log\PAR{\lambda^2R^2+r^2R^2-2\lambda
        R^2r\DOT{z}{u}}r^{\alpha-1}\sigma(\dd u)}\dd r
  \]
  for an arbitrary $z\in\dS^3$. Then, by the Funk\,--\,Hecke formula here again,
  \begin{align*}
    J(\lambda)
    &=-\tau_3\int_0^1\PAR{\int_{-1}^1\log\PAR{\lambda^2R^2+r^2R^2-2\lambda R^2rt}\sqrt{1-t^2}\,\dd t}r^{\alpha-1}\dd r,\\
    J'(\lambda)
    &=2\tau_3\int_0^1\PAR{\int_{-1}^1\frac{rt-\lambda}{\lambda^2+r^2-2\lambda rt}\sqrt{1-t^2}\,\dd t}r^{\alpha-1}\dd r.
  \end{align*}
  Now, if $0<\lambda<1$, then, still by using the same method,
  \begin{align*}
    J'(\lambda)
    &=2\tau_3\PAR{\int_0^{\lambda}+\int_{\lambda}^1}\PAR{\int_{-1}^1\PAR{\frac{rt-\lambda}{\lambda^2+r^2-2\lambda rt}}\sqrt{1-t^2}\,\dd t}r^{\alpha-1}\dd r\\
    &\quad +2\tau_3\int_\lambda^1\PAR{\frac{r}{r^2}\int_{-1}^1\PAR{\frac{t-\frac{\lambda}{r}}{\PAR{\frac{\lambda}{r}}^2+1-2\frac{\lambda}{r}t}}\sqrt{1-t^2}\,\dd
      t}r^{\alpha-1}\dd r\\
    &=\frac{2\tau_3}{\lambda}\int_0^\lambda\PAR{\int_{-1}^1\PAR{\frac{r}{\lambda}t-1}\sum_{n=0}^\infty
      U_n(t)\PAR{\frac{r}{\lambda}}^n\sqrt{1-t^2}\,\dd t}r^{\alpha-1}\dd r\\
    &\quad +2\tau_3\int_\lambda^1\PAR{\int_{-1}^1\PAR{t-\frac{\lambda}{r}}\sum_{n=0}^\infty
      U_n(t)\PAR{\frac{\lambda}{r}}^n\sqrt{1-t^2}\,\dd t}r^{\alpha-2}\dd r\\
    &=\frac{2}{\lambda}\int_0^\lambda\PAR{\PAR{\frac{r}{\lambda}}^2\frac{1}{2}-1}r^{\alpha-1}\dd
      r
      +2\int_\lambda^1\PAR{\frac{\lambda}{r}\frac{1}{2}-\frac{\lambda}{r}}r^{\alpha-2}\dd r\\
    &=\frac{8}{\alpha(\alpha-2)(\alpha+2)}\lambda^{\alpha-1}-\frac{1}{\alpha-2}\lambda.
  \end{align*}
  Thus, if $0<\lambda<1$, then
  \begin{align*}
    \varphi'(\lambda)
    &=I_1'(\lambda)+I_2'(\lambda)+\gamma\alpha R^\alpha\lambda^{\alpha-1}\\
    &=\frac{\alpha}{2}\Bigr(1-\frac{\alpha}{2}\Bigr) J'(\lambda)
      -\frac{\alpha}{4}\lambda
      +\gamma\alpha R^\alpha\lambda^{\alpha-1}\\
    &=-\frac{\frac{\alpha}{2}(1-\frac{\alpha}{2})}{\alpha-2}\lambda
      +\frac{\alpha}{2}\Bigr(1-\frac{\alpha}{2}\Bigr)
      \frac{8}{\alpha(\alpha-2)(\alpha+2)}\lambda^{\alpha-1}
      -\frac{\alpha}{4}\lambda+\gamma\alpha R^\alpha\lambda^{\alpha-1}\\
    &=\Bigr(\gamma\alpha R^\alpha-\frac{2}{\alpha+2}\Bigr)\lambda^{\alpha-1}.
  \end{align*}
  Now if we take $R=(\frac{2}{\gamma\alpha(\alpha+2)})^{\frac{1}{\alpha}}$,
  then $\varphi'(\lambda)=0$ for $0<\lambda<1$.

  We now consider $\lambda>1$. With $J(\lambda)$ as before, we have, using again
  the same method,
  \begin{align*}
    J'(\lambda)
    &=2\tau_3\int_0^1\PAR{\frac{\lambda}{\lambda^2}\int_{-1}^1\PAR{\frac{\frac{r}{\lambda}t-1}{1+\PAR{\frac{r}{\lambda}}^2-\frac{2r}{\lambda}t}}\sqrt{1-t^2}\,\dd
      t}r^{\alpha-1}\dd r\\
    &=2\tau_3\int_0^1\PAR{\frac{\lambda}{\lambda^2}\int_{-1}^1\PAR{\frac{r}{\lambda}t-1}\sum_{n=0}^\infty
      U_n(t)\PAR{\frac{r}{\lambda}}^n\sqrt{1-t^2}\,\dd
      t}r^{\alpha-1}\dd r\\
    &=\frac{1}{\lambda^3(\alpha+2)}-\frac{2}{\lambda\alpha}.
  \end{align*}
  Hence, using $R=\PAR{\frac{2}{\gamma\alpha(\alpha+2)}}^{\frac{1}{\alpha}}$,
  and the previously obtained values for $J'(\lambda)$ and $I_2'(\lambda)$, we
  get
  \begin{align*}
    \varphi'(\lambda)
    &=\beta J'(\lambda)+I_2'(\lambda)+\gamma\alpha R^\alpha\lambda^{\alpha-1}\\
    &=\frac{1}{4\lambda^3}\SBRA{\frac{\alpha(2-\alpha)}{\alpha+2}-4\lambda^2+\alpha+\frac{8}{\alpha+2}\lambda^{\alpha+2}}\\
    &=:\frac{G(\lambda)}{4\lambda^3}.
  \end{align*}
  We have $G(1)=\frac{\alpha(2-\alpha)+8}{\alpha+2}-4+\alpha=0$ while
  $G'(\lambda)=-8\lambda+8\lambda^{\alpha+1}=8\lambda(\lambda^\alpha-1)>0$ for
  $\lambda>1$. Hence $\varphi'(\lambda)>0$ for $\lambda>1$ and so
  $\varphi(\lambda)$ is increasing for $\lambda>1$. This ends the proof in the
  case $d=4$. Finally a careful examination of the proof reveals that it still
  works in the case $d\geq 5$ provided that we replace Chebyshev polynomials
  by Gegenbauer polynomials.
\end{proof}

\subsection{Case $d=3$, $s=d-4=-1$, and $1<\alpha<2$}\label{se:d=3etc}
Let $\mu$ be the probability measure on $\mathbb{R}^3$ parametrized by $R>0$,
and given by the mixture (convex combination)
\[
  \mu:=(2-\alpha)fm_3+(\alpha-1)\sigma_R,
\]
where
$f(x):=\frac{\alpha-1}{R^{\alpha-1}|\mathbb{S}^2|}|x|^{\alpha-4}\mathbf{1}_{|x|\leq
  R}$, $m_3$ is the Lebesgue measure on $\mathbb{R}^3$, and $\sigma_R$ is the
uniform probability measure on $\mathbb{S}^2_R$. Since $K_{-1}*\mu+V$ is
radially symmetric, for $x\in\mathbb{R}^3$, $x=\lambda R \xh$ with $\lambda\geq0$
and $\xh\in\sph$, we set
\begin{equation}\label{eq:d=3,s=-1,1<a<3}
  \varphi(\lambda):=(K_{-1}*\mu+V)(x)
  =-(2-\alpha)\int_{B_R} |x-y|f(y)\dd y
  -(\alpha-1)\int_{\sph} |x-R\yh|\sigma_1(\dd \yh)+\gamma R^\alpha\lambda^\alpha.
\end{equation}
The Frostman conditions are satisfied if we show that for some constant $c$,
we have $\varphi(\lambda)=c$ if $0\leq\lambda\leq1$ while
$\varphi(\lambda)\geq c$ if $\lambda\geq1$. Since $\varphi$ is continuous on
$[0,+\infty)$, and differentiable on $(0,+\infty)$, the desired result follows
from Lemma \ref{le:d=3,s=-1,1<a<2}.

\begin{lemma}\label{le:d=3,s=-1,1<a<2}
  Let $\varphi$ be as in \eqref{eq:d=3,s=-1,1<a<3} with $d=3$, $s=-1$,
  $1<\alpha<2$, and $R$ as in Theorem \ref{th:main}. Then
  \[
    \varphi'(\lambda)
    =\begin{cases}
      0 & \text{if $0<\lambda<1$}\\
      \geq0 & \text{if $\lambda>1$}
    \end{cases}.
  \]
\end{lemma}

\begin{proof}
  By the Funk\,--\,Hecke formula (Lemma \ref{le:funkhecke}),
  \begin{align*}
    \int|x-y|f(y)\dd y
    &=
      \frac{\alpha-1}{R^{\alpha-1}}\int_0^1\Bigr(\int_{\mathbb{S}^2}\sqrt{\lambda^2R^2-2R^2r\lambda
      z\cdot y-r^2R^2}\,\sigma_1(\dd y)\Bigr)R^{\alpha-1}r^{\alpha-2}\dd r\\
    &=(\alpha-1)\frac{R}{2}\int_0^1\Bigr(\int_{-1}^1\sqrt{\lambda^2-2r\lambda
      t+r^2}\dd t\Bigr)r^{\alpha-2}\dd r\\
    &=(\alpha-1)\frac{R}{6}\int_0^1\Bigr(\frac{(r+\lambda)^3-\left|r-\lambda\right| ^3}{\lambda}\Bigr)r^{\alpha-3}\dd r,
  \end{align*}
  while
  \begin{align*}
    \int|x-Ry|\sigma_1(\dd y)
    &=\int_{\mathbb{S}^2}\sqrt{\lambda^2R^2-2R^2\lambda z\cdot y+R^2}\,\sigma_1(\dd y)\\
    &=\frac{R}{2}\int_{-1}^1\sqrt{\lambda^2-2\lambda t+1}\,\dd t\\
    &=\frac{R}{6}\Bigr(\frac{(1+\lambda)^3-\left|1-\lambda\right|^3}{\lambda}\Bigr)
  \end{align*}
  which gives after some computations
  \[
    \varphi(\lambda)
    =\gamma R^\alpha\lambda^\alpha
    +R\begin{cases}
      \displaystyle
      \frac{1-\alpha }{(1+\alpha)\lambda}-\lambda
      & \text{if $\lambda\geq 1$}\\[1em]
      \displaystyle
      -\frac{2 \left(\alpha ^2+\lambda ^{\alpha }-1\right)}{\alpha  (\alpha +1)}
      & \text{if $\lambda\leq 1$}
    \end{cases};
  \]
  thus
  \[
    \varphi'(\lambda)
    =\gamma\alpha R^\alpha\lambda^{\alpha-1}
    +R\begin{cases}
      \displaystyle
      \frac{\alpha-1}{(\alpha+1)\lambda^2}-1
      & \text{if $\lambda>1$}\\[1em]
      \displaystyle
      -\frac{2}{\alpha+1}\lambda^{\alpha-1}
      & \text{if $0<\lambda<1$}
    \end{cases}.
  \]
  The condition $\varphi'(\lambda)=0$ when $0<\lambda<1$ forces
  $R=\Bigr(\frac{2}{\gamma\alpha(\alpha+1)}\Bigr)^{\frac{1}{\alpha-1}}$, and
  with this choice, for $\lambda>1$,
  \[
    \varphi'(\lambda)
    =R\Bigr(\frac{2}{\alpha+1}\lambda^{\alpha-1}
    -1+\frac{\alpha-1}{(\alpha+1)\lambda^2}\Bigr).
  \]
  We have
  \(
  \displaystyle
  \lim_{\lambda\to1^+}\varphi'(\lambda)
  =R\Bigr(\frac{2}{\alpha+1}-1+\frac{\alpha-1}{\alpha+1}\Bigr)
  =0,
  \)
  while for $\lambda>1$,
  \[
    \varphi''(\lambda)
    =2R\frac{\alpha-1}{\alpha+1}(\lambda^{\alpha-2}-\lambda^{-3})>0.
  \]
\end{proof}

\appendix

\section{Useful tools}

The (generalized) hypergeometric function, when it makes sense, is given by
\begin{equation}\label{eq:hype}
  \hgp{p}{q}(a_1,\ldots,a_p;b_1,\ldots,b_q;z)
  :=\sum_{k=0}^\infty
  \frac{(a_1)_k\cdots(a_p)_k}{(b_1)_k\cdots(b_q)_k}
  \frac{z^k}{k!},
\end{equation}
where $a_1,\ldots,a_p,b_1,\ldots,b_q,z\in\mathbb{C}$,
$(z)_k:=z(z+1)\cdots(z+k-1)$ is the Pochhammer symbol for rising factorial,
with convention $(z)_0:=1$ if $z\neq0$. If $\Re(z)>0$ then
$(z)_k=\Gamma(z+k)/\Gamma(z)$. The series is a finite sum when at least one of the $a_i$'s is a negative integer. It is undefined if
at least one of the $b_i$'s is a negative integer, and we exclude this
somewhat trivial situation from now on. If $p=q+1$ then the series converges if $|z|<1$. If $p<q+1$ then it converges for all $z$, while if $p>q+1$ then it diverges for all $z$ as soon as none of the $a_i$'s is negative integer. We primarily use $(p,q)=(2,1)$ (the Gauss hypergeometric function) and
$(p,q)=(3,2)$.

The following lemma states that $\hg$ appears as the series expansion of a certain Euler type integral, which follows essentially 
by using the binomial series expansion $(1-zu)^{-a}=\sum_{n=0}^\infty\binom{-a}{n}(-zu)^n$
together with classical Euler Beta integrals. This is useful for the handling
of certain of our integrals.

\begin{lemma}[Euler integral formula for $\hg$,
  see {\cite[Th.\ 2.2.1,~p.~65]{zbMATH01231230} or
    \cite[Eq.~15.3.1]{MR0167642}}]\label{le:hypeuler}
    \label{le:EulerIntHG}
  For all $a,b,c,z\in\mathbb{C}$ with $\Re(c)>\Re(b)>0$ and $|z|<1$,
  \[
    \int_0^1u^{b-1}(1-u)^{c-b-1}(1-zu)^{-a}\dd u
    =\frac{\Gamma(b)\Gamma(c-b)}{\Gamma(c)}
    \hg(a,b;c;z).
  \]
\end{lemma}

This formula allows $\hg(a,b;c;z)$ to be defined, by analytic continuation, for all $z\in\mathbb{C}\setminus[1,+\infty)$.
 Additionally, if $\Re(c - a - b) > 0$, then
  the series \eqref{eq:hype} for $\hg$ converges absolutely at $z = 1$ and
 \[
   \hg(a,b;c;1) = \frac{\Gamma(c) \Gamma(c-a-b)}{\Gamma(c-a) \Gamma(c-b)}.
 \]

Our main tool to reduce multivariate integrals into univariate integrals is
the Funk\,--\,Hecke formula, that gives the projection on any diameter of the uniform distribution on the sphere.

\begin{lemma}[Funk\,--\,Hecke formula, see {\cite[p.~18]{MR0199449}, \cite[Eq.~(5.1.9),~p.~197]{MR3970999}}]
  \label{le:funkhecke}
  Let $\sigma_1$ denote the uniform probability measure on $\sph$, $d\geq2$.
  Then, for all $\xh\in\sph$, 
  \[
    \int_{\sph}
    f(\xh\cdot \yh)\sigma_1(\dd \yh)
    =\tau_{d-1}\int_{-1}^1f(t)(1-t^2)^{\frac{d-3}{2}}\dd t,
  \]
  where
  \[
    \tau_{d-1}
    :=
    \PAR{\int_{-1}^1(1-t^2)^{\frac{d-3}{2}}\dd t}^{-1}
    =
    \frac{\Gamma(\frac{d}{2})}{\Gamma(\frac{1}{2})\Gamma(\frac{d-1}{2})}.
  \]
\end{lemma}

  In probabilistic terms, this means that if $Y$ is a random vector of
  $\mathbb{R}^d$ uniformly distributed on $\sph$ then for all
  $\xh\in\sph$, the law of $\xh\cdot Y$ has density
  $\tau_{d-1}(1-t^2)^{\frac{d-3}{2}}\ind_{t\in[-1,1]}$. This is an arcsine law
  when $d=2$, a uniform law when $d=3$, a semicircle law when $d=4$, and more
  generally, for arbitrary values of $d\geq2$, the image law by the map
  $u\mapsto\sqrt{u}$ of a beta law.

\begin{remark}[Scale invariance and homogeneous external
  field]\label{rm:scale}
  Let $\meq^\gamma$ be the equilibrium measure associated with $K_s$, $s>-2$,
  and $V=\gamma\ABS{\cdot}^\alpha$, $\alpha>0$, $\gamma>0$. In some sense,
  $\alpha$ is a shape parameter while $\gamma$ can be either a shape or scale
  parameter. Indeed, the homogeneity of $K_s$ and $V$ give
  \[
    \meq^\gamma
    =\mathrm{dilation}_{\gamma^{-1/\alpha}}(\meq^{\gamma^{-s/\alpha}}),
  \]
  and in particular if $s=0$ then
  $\meq^\gamma=\mathrm{dilation}_{\gamma^{-1/\alpha}}(\meq^1)$, where
  $\mathrm{dilation}_c(\mu)$ stands for the push forward of $\mu$ by the map
  $x\mapsto cx$. Such scaling properties play a role in various problems, see
  for instance Saff and Totik~\cite[Section IV.4]{MR1485778} and Hedenmalm and
  Makarov~\cite{MR3056295}.
\end{remark}

\begin{figure}[htbp]
  \centering
  \includegraphics[width=.7\textwidth,trim=2cm 8.5cm 2cm 8.5cm,clip]{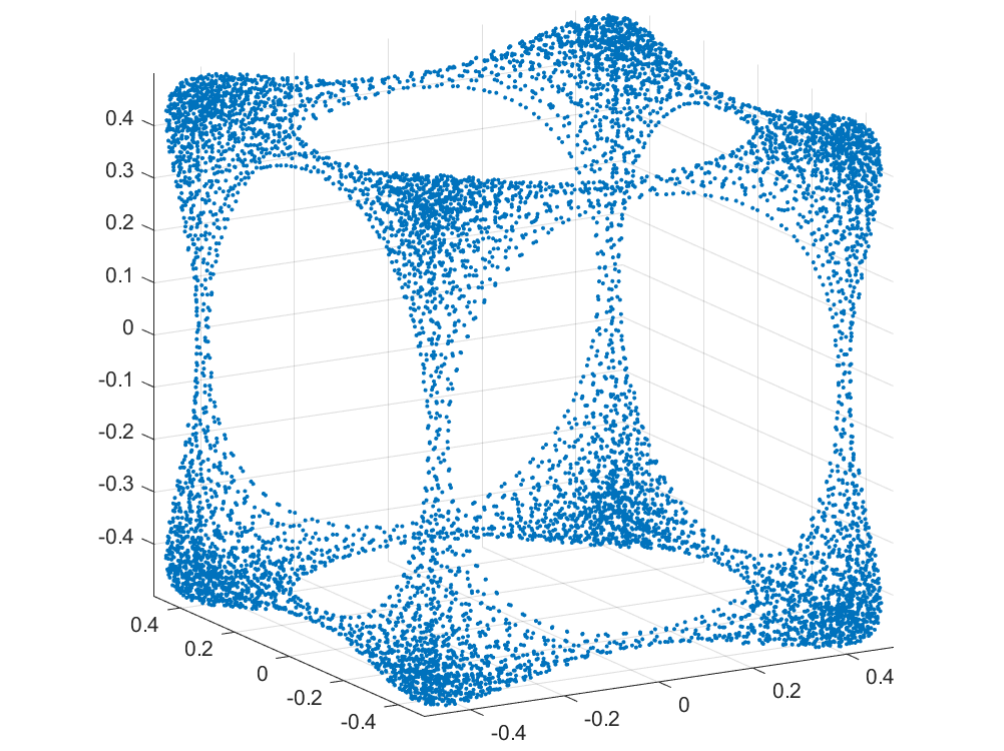}
  \caption{\label{fi:p}Projection on $x_1=0$ of a discrete numerical approximation of
    the equilibrium measure $\meq$ with $N = 10^4$ in the case where $s=0$,
    $d=4$, and $V=\ABS{\cdot}_4^4$.}
\end{figure}


\renewcommand{\MR}[1]{}%
\bibliographystyle{abbrv}%
\bibliography{hdep}%

\begin{thebibliography}{10}

\bibitem{MR0167642}
M.~Abramowitz and I.~A. Stegun.
\newblock {\em Handbook of mathematical functions with formulas, graphs, and
  mathematical tables}, volume~55 of {\em National Bureau of Standards Applied
  Mathematics Series}.
\newblock For sale by the Superintendent of Documents, U.S. Government Printing
  Office, Washington, D.C., 1964.

\bibitem{zbMATH01231230}
G.~E. {Andrews}, R.~{Askey}, and R.~{Roy}.
\newblock {\em {Special functions.}}, volume~71.
\newblock Cambridge: Cambridge University Press, 1999.

\bibitem{BalagueCarrilloLaurentRaoul2013}
D.~Balagu\'{e}, J.~A. Carrillo, T.~Laurent, and G.~Raoul.
\newblock Dimensionality of local minimizers of the interaction energy.
\newblock {\em Arch. Ration. Mech. Anal.}, 209(3):1055--1088, 2013.

\bibitem{MR2123199}
F.~Barthe, O.~Gu\'{e}don, S.~Mendelson, and A.~Naor.
\newblock A probabilistic approach to the geometry of the {$l^n_p$}-ball.
\newblock {\em Ann. Probab.}, 33(2):480--513, 2005.

\bibitem{BilykGMPV2021}
D.~Bilyk, A.~Glazyrin, R.~Matzke, J.~Park, and O.~Vlasiuk.
\newblock Energy on spheres and discreteness of minimizing measures.
\newblock {\em J. Funct. Anal.}, 280(11):Paper No. 108995, 28, 2021.

\bibitem{MR78470}
G.~Bj{\"{o}}rck.
\newblock Distributions of positive mass, which maximize a certain generalized
  energy integral.
\newblock {\em Ark. Mat.}, 3:255--269, 1956.

\bibitem{MR3970999}
S.~V. Borodachov, D.~P. Hardin, and E.~B. Saff.
\newblock {\em Discrete energy on rectifiable sets}.
\newblock Springer Monographs in Mathematics. Springer, New York, 2019.

\bibitem{ByrdLuNocedalZhu1995}
R.~H. Byrd, P.~Lu, J.~Nocedal, and C.~Y. Zhu.
\newblock A limited memory algorithm for bound constrained optimization.
\newblock {\em SIAM J. Sci. Comput.}, 16(5):1190--1208, 1995.

\bibitem{CanizoCarrilloPatacchine2015}
J.~A. Ca\~{n}izo, J.~A. Carrillo, and F.~S. Patacchini.
\newblock Existence of compactly supported global minimisers for the
  interaction energy.
\newblock {\em Arch. Ration. Mech. Anal.}, 217(3):1197--1217, 2015.

\bibitem{CarrilloFigalliPatacchini2017}
J.~A. Carrillo, A.~Figalli, and F.~S. Patacchini.
\newblock Geometry of minimizers for the interaction energy with mildly
  repulsive potentials.
\newblock {\em Ann. Inst. H. Poincar\'{e} Anal. Non Lin\'{e}aire},
  34(5):1299--1308, 2017.

\bibitem{MR3262506}
D.~Chafa\"\i, N.~Gozlan, and P.-A. Zitt.
\newblock First-order global asymptotics for confined particles with singular
  pair repulsion.
\newblock {\em Ann. Appl. Probab.}, 24(6):2371--2413, 2014.

\bibitem{potspe}
D.~Chafa\"{\i}, E.~B. Saff, and R.~S. Womersley.
\newblock On the solution of a {R}iesz equilibrium problem and integral
  identities for special functions.
\newblock {\em J. Math. Anal. Appl.}, 515(1):Paper No. 126367, 2022.

\bibitem{Choquet1958}
G.~Choquet.
\newblock Diam{\`e}tre transfini et comparaison de diverses capacit{\'e}s.
\newblock Technical report, Facult{\'e} des Sciences de Paris, 1958.

\bibitem{NIST:DLMF2022}
{\it NIST Digital Library of Mathematical Functions}.
\newblock \url{http://dlmf.nist.gov/}, Release 1.1.5 of 2022-03-15.
\newblock F.~W.~J. Olver, A.~B. {Olde Daalhuis}, D.~W. Lozier, B.~I. Schneider,
  R.~F. Boisvert, C.~W. Clark, B.~R. Miller, B.~V. Saunders, H.~S. Cohl, and
  M.~A. McClain, eds.

\bibitem{MR3640641}
B.~Dyda, A.~Kuznetsov, and M.~Kwa\'{s}nicki.
\newblock Fractional {L}aplace operator and {M}eijer {G}-function.
\newblock {\em Constr. Approx.}, 45(3):427--448, 2017.

\bibitem{Fekete1923}
M.~Fekete.
\newblock \"{U}ber die {V}erteilung der {W}urzeln bei gewissen algebraischen
  {G}leichungen mit ganzzahligen {K}oeffizienten.
\newblock {\em Math. Z.}, 17(1):228--249, 1923.

\bibitem{GutlebCarrilloOlver2022Balls}
T.~S. Gutleb, J.~A. Carrillo, and S.~Olver.
\newblock Computation of power law equilibrium measures on balls of arbitrary
  dimension.
\newblock {\em Constr. Approx.}, 2022.
\newblock
  \href{https://doi.org/10.1007/s00365-022-09606-0}{DOI:10.1007/s00365-022-09606-0}.

\bibitem{GutlebCarrilloOlver2021}
T.~S. Gutleb, J.~A. Carrillo, and S.~Olver.
\newblock Computing equilibrium measures with power law kernels.
\newblock {\em Math. Comp.}, 91(337):2247--2281, 2022.

\bibitem{MR3056295}
H.~Hedenmalm and N.~Makarov.
\newblock Coulomb gas ensembles and {L}aplacian growth.
\newblock {\em Proc. Lond. Math. Soc. (3)}, 106(4):859--907, 2013.

\bibitem{hertrich-graf-beinert-steidl}
J.~Hertrich, M.~Gräf, R.~Beinert, and G.~Steidl.
\newblock Wasserstein steepest descent flows of discrepancies with {R}iesz
  kernels.
\newblock preprint \href{https://arxiv.org/abs/2211.01804}{arXiv:2211.01804},
  2022.

\bibitem{MR1746976}
F.~Hiai and D.~Petz.
\newblock {\em The semicircle law, free random variables and entropy},
  volume~77 of {\em Mathematical Surveys and Monographs}.
\newblock American Mathematical Society, Providence, RI, 2000.

\bibitem{MR0350027}
N.~S. Landkof.
\newblock {\em Foundations of modern potential theory}.
\newblock Springer, 1972.
\newblock Translated from the Russian by A. P. Doohovskoy, Die Grundlehren der
  mathematischen Wissenschaften 180.

\bibitem{MR2647570}
A.~L\'{o}pez~Garc\'{\i}a.
\newblock Greedy energy points with external fields.
\newblock In {\em Recent trends in orthogonal polynomials and approximation
  theory}, volume 507 of {\em Contemp. Math.}, pages 189--207. Amer. Math.
  Soc., Providence, RI, 2010.

\bibitem{MR0199449}
C.~M{\"{u}}ller.
\newblock {\em Spherical harmonics}, volume~17 of {\em Lecture Notes in
  Mathematics}.
\newblock Springer, 1966.

\bibitem{NocedalWright2006}
J.~Nocedal and S.~J. Wright.
\newblock {\em Numerical optimization}.
\newblock Springer Series in Operations Research and Financial Engineering.
  Springer, New York, second edition, 2006.

\bibitem{MR1485778}
E.~B. Saff and V.~Totik.
\newblock {\em Logarithmic potentials with external fields}, volume 316 of {\em
  Die Grundlehren der mathematischen Wissenschaften}.
\newblock Springer, 1997.
\newblock Appendix B by Thomas Bloom.

\bibitem{MR2498715}
F.~Sinz, S.~Gerwinn, and M.~Bethge.
\newblock Characterization of the {$p$}-generalized normal distribution.
\newblock {\em J. Multivariate Anal.}, 100(5):817--820, 2009.

\end{thebibliography}
\end{document}